\documentclass[11pt, english,english]{article}
\usepackage{amsmath,mathrsfs} 
\usepackage{amssymb} 
\usepackage{amsfonts}
\usepackage{bbm}
\usepackage{mathtools}
\usepackage{amsthm}
\newtheorem{definition}{Definition}

\newtheorem{theorem}{Theorem}
\usepackage{hyperref}

\usepackage{imakeidx}

\usepackage[english]{babel}

\newcommand{\inR}{\in \mathbb{R}}

\newcommand{\C}{ \mathbb{C}}
\newcommand{\R}{ \mathbb{R}}
\newcommand{\Z}{ \mathbb{Z}}

\newcommand{\N}{ \mathbb{N}}

\newcommand{\eqdef}{\stackrel{\vartriangle}{=}}

\newcommand{\Lop}{{\rm L}}
\newcommand{\Dop}{{\rm D}}

\newcommand{\dint}{{\rm d}}

\newcommand{\bw}{{\boldsymbol \omega}}

\def\V#1{{\boldsymbol{#1}}}         
\def\Spc#1{{\mathcal{#1}}}  
\def\M#1{{\bf{#1}}}  
\def\Op#1{{\mathrm{#1}}}  
\def\ee{\mathrm{e}} 
\def\jj{\mathrm{i}}

\def\Prob{\mathscr{P}}
\def\Form{\widehat{\Prob}} 
\def\Exp{\mathop{\mathbb{E}}\nolimits}

\def\Identity{\mathrm{Id}} %

\usepackage{mathrsfs,bm,enumerate}
\usepackage{mathtools}

\newcommand{\embedC}{\xhookrightarrow{}}
\newcommand{\embedD}{\xhookrightarrow{\mbox{\tiny \rm d.}}}

\renewcommand{\[}{\begin{equation}}
\renewcommand{\]}[1]{\label{eq:#1}\end{equation}}

\def\Prob{\mathscr{P}}
\def\Form{\widehat{\Prob}} 
\def\Exp{\mathop{\mathbb{E}}\nolimits}

\def\wRand{\xi}
\def\bwRand{{\boldsymbol \xi}}
\def\Rand#1{{\uppercase{#1}}}         
\def\VRand#1{{\uppercase{\boldsymbol #1}}}         
\def\gRand#1{{\uppercase{ #1}}}         

\newcommand{\Tr}{\mathsf{T}}
\newcommand{\HTop}{\mathsf{H}}

\usepackage{longtable}
\usepackage{wrapfig}

\usepackage{lipsum}

\begin{document}

\title{Generalized Splines and Gaussian Processes}
\author{
Michael Unser
\thanks{Biomedical Imaging Group, \'Ecole polytechnique f\'ed\'erale de Lausanne (EPFL),
Station 17, CH-1015, Lausanne, Switzerland ({\tt michael.unser@epfl.ch}). }
 }

\maketitle

\begin{abstract}
For finite-dimensional linear inverse problems where the variables are Gaussian, it is well-known that the minimum-mean-square error estimator takes the form of a regularized least-squares data fit.  In this chapter, we show that this equivalence extends to a much broader infinite-dimensional setting where generalized splines take the role of linear regressors and generalized Gaussian processes on a nuclear space $\Spc S$ are the counterpart of Gaussian random vectors. The scope of this extension is of the same nature as the switch from the classic notion of function to that of a distribution, also known as a ``generalized function.''
Our formalism involves a whitening/regularization operator $\Op L: \Spc S\to \Spc S'$ whose continuous extension induces a native Hilbert space $\Spc H\subset \Spc S'$ that plays a central role in our characterization. The presentation is self-contained for the most part and remarkably general and powerful. It allows for the recovery of all known instances of such equivalences; in particular, the methods
involving innovations and reproducing-kernel Hilbert spaces developed by Kailath and his students, and the mathematical correspondence between fractional splines and Mandelbrot's fractional Brownian motion (fractals), with the former being the optimal estimators of the latter. It also covers general Bayesian methods for the resolution of infinite-dimensional inverse problems.
\end{abstract}

\newpage
\tableofcontents

\newpage
\section{Introduction}
The algorithmic correspondence between optimal estimation under the multivariate Gaussian hypothesis and regularized least-squares regression (RLS) is a central theme in estimation theory \cite{Kailath2000linear}. 
Specifically, RLS can implement the maximum-a-posteriori (MAP) estimator by selecting a regularization functional that matches the prior Gaussian log-likelihood of the signal to be estimated, and---since the posterior is also Gaussian---the MAP and minimum-mean-square-error (MMSE) estimators coincide. Conversely, any quadratic regularizer admits a ``reversed-engineered'' interpretation as a Gaussian log-likelihood.
This type of equivalence extends to infinite dimensions under suitable conditions (e.g., continuity of the underlying functions, stationarity) \cite{Wahba1990,Berlinet2004}, with the caveat
that the prior log-likelihood (and, hence, the MAP estimator) of an infinite-dimensional random variable is no longer well-defined.

To formulate an equivalence in the continuum, one needs a function-analytic counterpart of regularized least-squares---namely, a method to fit (finite) data with a function---which is precisely what splines do \cite{Schumaker1981}. The relevant splines here are the variational ones defined as minimizers of a spline energy subject to interpolation constraints \cite{Prenter:1975}. The relaxation of the constraints  yields a smoothing spline, a concept that can be traced back to Schoenberg \cite{Schoenberg1964}, the founding father of splines. In 1966, Carl de Boor 
introduced an elegant and remarkably powerful formalization of variational splines \cite{deBoor1966} that involves reproducing-kernel Hilbert spaces (RKHS) \cite{Aronszajn1950}.

The infinite-dimensional counterparts of Gaussian random vectors are the Gaussian stochastic processes (GSP)---discrete or continuous---which play a central role in time-series analysis, statistical signal processing, control theory, and machine learning. A significant part of the theory of GSP revolves around the use of RKHS, which was initiated by Parzen \cite{Parzen1961}. 
Another key notion is the innovations representation (IR) \cite{Kailath1970}, which allows us to interpret a GSP as a filtered white noise.
By exploiting the interplay between RKHS and IR \cite{Kailath1972}, Kailath and collaborators pioneered the use of IR to develop efficient estimation and detection methods, including Kalman-type filters, thereby laying the foundations of modern statistical signal processing and control \cite{Kailath1974}.

The link between splines and optimal (linear) estimation was uncovered around the same time by Kimeldorf and Wahba, first under the hypothesis of stationary \cite{Kimeldorf1970} and then for more general Gaussian processes defined on an interval  \cite{Kimeldorf1971}.
Kailath and his student Howard Weinert (PhD 1970, Stanford) expanded the framework by introducing more general splines and recursive algorithms \cite{Weinert1974}, motivated by applications in signal processing and optimal control \cite{Weinert1976}.
Wahba's formalization of the statistical optimality of splines was very impactful; it motivated the extension of the framework to higher dimensions \cite{Wahba1990, Myers1992} and, more, generally, the very successful use of kernel methods in statistics and classic machine learning \cite{Berlinet2004,Hofmann2008}.

In this chapter, we present
a unifying account of these various topics from the higher-level perspective of operators acting on nuclear spaces.
To reassure the readers who are not familiar with these notions, our intent is to recover the formal simplicity of the finite-dimensional setting. For reference purposes, the latter is reviewed in the appendix, which may also serve as a roadmap for the chapter, with a concrete translation of the main results of the theory.
While our treatment requires some level of abstraction, it allows us to limit the mathematical assumptions to the bare minimum, with the following payoffs. \begin{itemize}
\item The generality and completeness of the formulation, which is inherited from the use of an abstract nuclear space $\Spc S$ in the statement of all main results.
This covers the basic textbook setting with  $\Spc S=\R^N$, discrete time signals with $\Spc S=\Spc S(\Z)$, periodic signals as in \cite{Badoual2018}
with $\Spc S=\Spc S(\mathbb{T})=C^\infty(\mathbb{T})$ and $\mathbb{T}=[0,2\pi]$,
continuous-domain signals with $\Spc S=\Spc S(\R)$, multidimensional signals (either continuous or discrete) with $\Spc S=\Spc S(\R^d)$
(resp., $\Spc S=\Spc S(\Z^d)$), and many more. 
\item The use of linear operators as alternative to reproducing kernels: In the continuum, this enables the consideration of a much broader class of objects than ordinary functions such as, for instance, white noise on the side of the stochastic processes, and ``pseudo-discrete'' signals that are linear combinations of shifted replicates of a generalized kernel (including the case of the Dirac impulse) on the side of the splines.
\item The consideration of a probing mechanism that involves continuous linear functionals---the generalization of observing the sample-values of a function---in the spirit of the $\Lop$g splines of Jerome and Schumaker \cite{Jerome1969}.
This makes the framework suitable not only for interpolation and machine learning, but also for the resolution of linear inverse problems such as the reconstruction of biomedical images \cite{Gupta2018,Unser_2020}.
\item The unified handling of quasi-invertible whitening/regularization operators, which are those that have a nontrivial, finite-dimensional null space. This results in a congruence between the {\em native space} of a given type of splines and a family of Gaussian processes that fulfill specific boundary conditions. The fact that such processes are necessarily non-stationary results in a significant extension of the classic equivalences uncovered by Wahba et al.\ \cite{Kimeldorf1970}. The canonical example that falls in this ``extended'' category is the Brownian motion $\Rand B(t)$. This process satisfies the boundary condition $\Rand B(0)=0$ and is whitened by $\Lop=\Dop=\frac{\dint}{\dint t}$: the operator that induces piecewise-linear splines. This allows us to recover L\'evy's classical result \cite{Levy1934} that predates both splines and the theory of stochastic processes: the MMSE estimator of $\Rand B(t)$ for $t\in[t_1,t_2]$ given the sample values $\Rand B(t_1)=b_1$ and $\Rand B(t_2)=b_2$ is the straight line that connects $b_1$ and $b_2$.
\end{itemize}
The chapter is organized as follows: Section \ref{Sec:Foundations} is devoted to the exposition of the mathematical framework, starting from ``conventional'' Hilbert spaces and then moving on to nuclear spaces. The key result
is Theorem \ref{Theo:conditionPos}, which specifies the ``native'' Hilbert spaces that will be our substitute for the RHKS of the classic theory. In Section \ref{Sec:Deterministic}, we develop the deterministic theory of generalized splines associated with a ${\V p}$-admissible operator $\Lop$, with our Definition \ref{Def:pAdmissibility} of admissibility being the least constraining one to date. The native spaces for these splines are characterized in Theorem \ref{Theo:FactoRiesz} with the primary part of the space being isometric\footnote{This property is fundamental as it also guarantees the existence of a corresponding innovations representation in the stochastic counterpart of the theory.} to $L_2$.  The characterization of the corresponding splines then easily follows via the use of standard Hilbert-space techniques.
In Section \ref{Sec:GGP}, we use the same operators and functional tools to specify generalized Gaussian processes (GGP).
The characterization and the proof of existence of these objects (Theorem \ref{Theo:GGaussProcess}) heavily relies on Gelfand and Vilenkin's theory of generalized stochastic processes
\cite{Gelfand-Villenkin1964} and the use of the characteristic functional---the infinite-dimensional generalization of the characteristic function of probabilists (see also \cite{Unser2014book}).
We then close the circle by showing that the MMSE estimator of a GGP from a given collection of (noisy) linear measurements is a generalized spline with a regularization operator $\Lop$ that matches the whitening operator of the process.
\section{Mathematical Foundations}
\label{Sec:Foundations}
\subsection{Hilbert Spaces and Riesz Conjugates}
\begin{definition}[Semi-inner product] 
Let $\Spc H$ be a (real-valued) vector space. 
Then, the map $\langle \cdot,\cdot \rangle_\Spc H: \Spc H \times \Spc H \to \R$ is called a {\em semi-inner product} if it satisfies the three following properties for any $f, f_1,f_2 \in \Spc H$ and $\alpha \in \R$. 
\begin{enumerate}
\item Linearity: $\langle \alpha f_1 + f_2, f \rangle_{\Spc H}=\alpha \langle f_1, f \rangle_{\Spc H} +\langle f_2, f \rangle_{\Spc H}$.
\item 
Symmetry:  $\langle f_1,f_2\rangle_{\Spc H}=\langle f_2,f_1\rangle_{\Spc H}.$
\item Semi-positive definiteness: $|f|_{\Spc H}^2=\langle f,f\rangle_{\Spc H}\ge 0$.
\end{enumerate}
If the third condition can be replaced by the
stricter variant
\begin{itemize}
\item positive-definiteness: $\langle f,f\rangle_{\Spc H}> 0$  for all $f \ne 0$,
\end{itemize}
then $\langle \cdot,\cdot \rangle_\Spc H$ is an {inner product}, in which case
the induced seminorm $|f|_{\Spc H}$ becomes a Hilbertian norm denoted by $\|f\|_{\Spc H}=\sqrt{\langle f,f\rangle_{\Spc H}}$.
\end{definition}
A Hilbert space is a complete vector space equipped with a norm induced by an inner product. It is denoted by $(\Spc H,\|\cdot\|_{\Spc H})$ or, simply, by $\Spc H$ (for short). The Hilbert space $\Spc H$ has a unique topological dual
$\Spc H'$, which is itself a Hilbert space equipped with the dual norm $\|\cdot\|_{\Spc H'}$ (see \eqref{Eq:DualNorm} below).
Formally, an element $\nu$ of the dual space $\Spc H'$ is a continuous linear functional $\nu: \Spc H \to \R$. 
Likewise, since all Hilbert spaces are reflexive, we can view any element $f \in \Spc H=\Spc H''=(\Spc H')'$ as a continuous linear functional $f: \Spc H'\to \R$.
The bilateral character of this association is described by the {\em duality product}
\begin{align}
\Spc H' \times \Spc H \to \R: (\nu,f) \mapsto \langle \nu, f\rangle_{\Spc H' \times \Spc H}=\langle f,\nu\rangle_{\Spc H \times \Spc H'}\in \R,
\end{align}
which is a map that is linear and continuous in both arguments. While the duality pairing and the inner product $\langle \cdot,\cdot\rangle_\Spc H$ are both bilinear forms, they are fundamentally distinct because the former connects elements living in different complementary spaces with the concept of duality pairing also applying to more general topological spaces such as Banach and nuclear spaces \cite{Rudin1991}.

To avoid notational overload, we shall henceforth drop the subscript in the specification of the duality product, with the understanding that the first argument is a linear functional that acts on the second argument, as in  $\nu: f \mapsto \langle \nu, f\rangle$, where $f \in \Spc H$ usually also has a concrete identification as a vector or a function. The specification of these linear functionals then naturally leads to the  definition of the 
dual norm
\begin{align}
\label{Eq:DualNorm}
\|\nu\|_{\Spc H'}\eqdef\sup_{f\in\Spc H\backslash\{0\}} 
\frac{\langle \nu,f\rangle}{\|f\|_{\Spc H}}. 
\end{align}

A fundamental result in the theory of Hilbert spaces is that $\Spc H'$ is isometrically isomorphic to $\Spc H$. Specifically, there exists a unitary transform $\Op J_{\Spc H'}: \Spc H' \to \Spc H$ (the so-called {\em Riesz map}) such that
\begin{align}
\forall \nu_1,\nu_2 \in \Spc H': \langle \nu_1, \nu_2\rangle_{\Spc H'}=\langle \Op J_{\Spc H'}\{\nu_1\},\Op J_{\Spc H'}\{\nu_2\}\rangle_{\Spc H}=\langle \nu^\ast_1, \nu^\ast_2\rangle_{\Spc H}.
\label{Eq:Rieszequival}
\end{align}
Conversely, we may define the Riesz map by setting $\Op J_{\Spc H'}\{\nu\}=\nu^\ast$, where $\nu^\ast$ (the {\em Riesz conjugate} of $\nu \in \Spc H'$) is the unique element of $\Spc H$ such that \eqref{Eq:Rieszequival} with $\nu=\nu_1=\nu_2$ holds.
\begin{theorem}[Riesz map and conjugates]
\label{Def:DualMap}
Let $(\Spc H,\Spc H')$ be a dual pair of Hilbert spaces.
\begin{enumerate}
\item Isomorphism: There exists a unique linear unitary map $\Op J_{\Spc H}=\Op J_{\Spc H'}^{-1}: \Spc H \to \Spc H'$ such that $\Spc H'=\Op J_{\Spc H}(\Spc H)$ and $\Spc H=\Op J^{-1}_{\Spc H}(\Spc H')$.
\item Unicity of Riesz conjugate: Any $\nu \in \Spc H'$ has a unique representer
$\nu^\ast=\Op J_{\Spc H'}\{\nu\} \in \Spc H$ such that $\|\nu\|_{\Spc H'}=\|\nu^\ast\|_{\Spc H}$ (isometry) and $\langle \nu,\nu^\ast \rangle=\|\nu\|^2_{\Spc H'}=\langle \nu,\nu \rangle_{\Spc H'}.$
\item From duality to inner products: For any $(\nu,f) \in \Spc H' \times \Spc H$, we have that 
$$\langle \nu,f \rangle= \langle \nu^\ast ,f \rangle_{\Spc H}=\langle \nu ,f^\ast \rangle_{\Spc H'}  \quad \mbox{ with }\quad f^\ast=\Op J_{\Spc H}\{f\}.$$
\end{enumerate}
\end{theorem}
Note that the first property implies invertibility with $\Op J_{\Spc H'}^{-1}=\Op J_{\Spc H}: \Spc H \to \Spc H'$ or, equivalently, $\nu^{\ast\ast}=\nu$  for all $\nu \in \Spc H'$.  

Since Hilbert spaces are reflexive, the role of $\Spc H=(\Spc H')'$ and $\Spc H'$ in Theorem \ref{Def:DualMap}
is interchangeable. 
The canonical example is $(\Spc H, \Spc H')=\big(L_2,L_2\big)$ with $\Op J_{L_2}=\Identity$
and $L_2=\ell_2(\Z)$ or $L_2=L_2(\R)$. 

\subsection{Extended Nuclear-Space Framework}
To bypass some of the technicalities of classic functional analysis and to offer the greatest level of generality, we are considering an abstract formulation that involves a dual pair $(\Spc S, \Spc S')$ of complete nuclear spaces, with the objects of interest living in a Hilbert space $\Spc H$ with the property that $ \Spc S\embedD \Spc H \embedD \Spc S'$, where the symbol ``$\embedD$'' indicates a topological embedding that is both dense and continuous. We also require that $\Spc S$ be canonically embedded in $\Spc S'$ so that every test function $\varphi$ can be identified as a
distribution with $\varphi \in \Spc S': \phi \mapsto \langle \varphi, \phi\rangle$, $\phi \in \Spc S'$. This configuration has two important implications: (i) it gets transferred to the dual spaces with  
$\Spc S''=\Spc S \embedD \Spc H' \embedD \Spc S'$ (by duality); and (ii) it enables the specification of the Hilbert spaces of our interest by completion: $\Spc H=\overline{(\Spc S,\|\cdot\|_{\Spc H})}$
and $\Spc H'=\overline{(\Spc S,\|\cdot\|_{\Spc H'})}$.
This construction mechanism is our alternative to the classic use of reproducing kernels, and we shall see that it is much more powerful.

While the terminology ``nuclear'' may be unfamiliar to most readers, they must have been exposed to at least two nuclear spaces: $\R^N$ (for linear algebra) and $\Spc S'(\R)$ (Schwartz' space of tempered distributions) which, for instance, allows for the rigorous definition of the Dirac impulse $\delta$ and its derivatives. This category of spaces was revealed by the famous mathematician Grothendieck, who devoted his PhD (under the supervision of L. Schwartz) to the clarification and extension of the theory of locally convex topological spaces \cite{Grothendieck1955}. 
The bottom line for practitioners is that the analysis methods for nuclear spaces are the same as in distribution theory with the only critical issue being 
to ensure the continuity of the underlying transformations (linear operators and functionals).
The payoff is that it simplifies many derivations while providing a much higher level of generality than the classic Lebesgue's approach to functional analysis.

The precise definition of a nuclear space is rather abstract (which explains why they are under-used) and beyond the scope of this chapter. We therefore focus on the part that is relevant  to our purpose: the unique and powerful functional properties associated with the framework \cite{Horvath1966topological,Treves2006}.
\begin{enumerate}
\item Complete nuclear spaces are reflexive with $(\Spc S')'=\Spc S$ and with their dual being nuclear as well.
\item Nuclear spaces have a special locally convex topology induced by an infinite family of seminorms.
\item Montel property: The nuclear topology induces an equivalence between strong and weak($\ast$) convergence, as if we were in finite dimensions.
\item  Nuclear spaces allow for an exhaustive characterization (the generalization of Schwartz' famous kernel theorem) of all continuous operators $\Lop:  \Spc S \to \Spc S'$.
\item The dual space $\Spc S'$ is large enough to contain all the (infinite-dimensional) Hilbert or Banach spaces of interest to
analysts, while the predual $\Spc S$---the space of ``test'' functions---is typically small enough to be included in all of them, as exemplified by the  relation $\Spc S \subset \Spc H=\overline{(\Spc S,\|\cdot\|_{\Spc H})} \subset \Spc S'$.
\end{enumerate} 
Properties 3-5 are remarkable because they hold neither in Hilbert nor in Banach spaces unless the dimension is finite (i.e., $\Spc S=\R^N$).
Property 3, in particular, makes the investigation of continuity and convergence much easier than in classic functional analysis.
Properties 4 and 5 enable the formulation of ``complete'' theories (as the present one), with no more room for extensions, which is the reason why theoreticians sometimes say that Grothendieck killed research in functional analysis. \\[1ex]
{\em The Pivot Space $L_2=(L_2)'$}: Any given dual pair $(\Spc S,\Spc S')$
 of nuclear spaces induces a canonical Hilbert space $L_2$ defined as
\begin{align}
\label{Eq:L2}
L_2=\overline{(\Spc S,\|\cdot\|_2)}  \mbox{ with } \|\varphi\|_2=\sqrt{\langle\varphi,\varphi \rangle} \mbox{ for any } \varphi \in \Spc S.
\end{align} 
This pivot space, which is central to our formulation, is equal to its own dual because its topology is inherited from the $(\Spc S,\Spc S')$-duality product. For instance, the concrete space defined by \eqref{Eq:L2} with $\Spc S=\Spc S(\R)$ is precisely $L_2(\R)$ (Lebesgue's space of finite-energy functions). The present definition is remarkable for its conciseness and the fact it does not require any knowledge of the Lebesgue integral. This equivalence can be taken as a perfect example of how the higher-level of abstraction of nuclear spaces induces ``formal simplicity.''
\\[1ex]
{\em Link with RKHS}: The concrete nuclear spaces of our interest are of the form $\Spc S=\Spc S(\mathbb{I})$ with $\mathbb{I}$ a set such as $\{1,\dots,N\}$ (for vectors), $\Z$ (for discrete signals), $\mathbb{T}=[0,2\pi]$ (for periodic functions), $\R^d$ (for multidimensional functions), $\Spc S=\Spc S(\mathbb{S}^{d-1})$ (for working on the (hyper-)sphere), among others. This signifies that any $\varphi \in \Spc S$ is identifiable as a function $\varphi: \mathbb{I}\to \R$ on the domain $\mathbb{I}$, with the nuclear topology imposing that it be ``smooth'' and ``rapidly decaying''. Correspondingly, for any $x_0 \in \mathbb{I}$, there exists an evaluation functional $\delta_{x_0}: \varphi \mapsto \langle \delta_{x_0},\varphi\rangle=\varphi(x_0)$ (the equivalent of the Dirac impulse shifted by $x_0$) that is included in $\Spc S'(\mathbb{I})$. These functionals are the formal mechanism for sampling (or evaluating) the test functions at any location $x_0$. By continuity, this behavior extends to a subclass of Hilbert spaces $\Spc H \subset \Spc S'(\mathbb{I})$, which coincides with the classic family of RKHS. The members of RKHS are ordinary {\em continuous} functions, as opposed to $L_2$-functions and, more generally, distributions.

%
\subsection{From Nuclear to Hilbert Spaces}
\label{Sec:NtoHilbert}
Notation: We use $\V \nu=(\nu_1,\dots,\nu_{M})$ to denote a collection of linear functionals with $\nu_n \in \Spc S'$ (or $\Spc H'$), $m=1, \dots, M$.
Under the implicit assumption that the $\nu_m$ are linearly independent, we then define: 
\begin{enumerate}
\item the linear subspace
$\Spc N_{\V \nu}={\rm span}\{\nu_m\}_{m=1}^{M} \subset \Spc S'$;
\item the linear vector-functional
\begin{align}
\V \nu: \varphi \mapsto \V \nu(\varphi)=\big (\langle \nu_1, \varphi\rangle,\dots, \langle \nu_{M}, \varphi\rangle\big):  \Spc S \to \R^{M};
\end{align}
\item the symmetric, finite-rank operator $\Op R_{\V \nu}: \Spc S \to \Spc N_{\V \nu}\subset \Spc S'$ (or  $\Op R_{\V \nu}: \Spc H \to \Spc N_{\V \nu}\subset \Spc H'$) with
\begin{align}
\Op R_\V \nu: \varphi \mapsto \Op R_\V \nu\{\varphi\}=\sum_{m=1}^M\nu_m \langle \nu_m, \varphi\rangle.
\end{align}
\end{enumerate}
Our abstract definition of splines involves a specific $\V \nu: \Spc H \to \R^M$
and a regularization operator $\Lop: \Spc H \to L_2$ where $\Spc H=\Spc H_\Lop \subset \Spc S'$ is the so-called {\em native space} to be identified in Section \ref{Sec:NativeSpaces}. These generalized splines are parameterized by a vector $\M f\in \R^M$ (typically, a data point) and are the (unique) solution of the optimization problem
\begin{align}
\label{Eq:GsplineInt}
\min_{f \in \Spc H} \|\Lop f\|_{L_2} \mbox{ s.t. } \V \nu(f)=\M f.
\end{align}
The most interesting (and also more challenging) forms of this problem are the ones where $\Lop$ has a nontrivial null space
$\Spc N_{\Lop}=\Spc N_{\V p}$ that is spanned by a finite-dimensional basis $\V p=(p_1,\dots,p_{N_0})$ with
$p_n \in \Spc S'$ and $0 < N_0 <M$. 

As first step, we show how $\Spc N_{\Lop}$ can be turned into a Hilbert space, the non-standard part being that the
$p_n$ (which are typically Taylor monomials) are not included in $L_2$, so that they cannot be orthogonalized. Our solution is to 
rely on some biorthogonal basis $\V \phi=(\phi_1,\dots,\phi_{N_0})$.
The latter is generally not unique unless
one forces the $\phi_n$ to be part of $\Spc N_{\V p}$, which is only possible if $\Spc N_{\V p}\subset L_2$.

\begin{theorem}[Finite-dimensional Hilbert subspaces of $\Spc S'$]
\label{Theo:FiniteDimHilbert}
Let $\V p=(p_n)$ be a basis of $\Spc N_{\V p}={\rm span}\{p_n\}_{n=1}^{N_0}\subset \Spc S'$. Then, there exits
 a vector of functionals $\V \phi=(\phi_n)$ with $\phi_1,\dots,\phi_{N_0}\in \Spc S$ such that
 $\langle p_n, \phi_m\rangle = \delta_{m-n}$ (biorthogonality).
 
The resulting biorthogonal system $(\V p, \V \phi)$ is called {\em universal} and it allows us to identify the dual pair of 
Hilbert spaces  $\Spc H_0=(\Spc N_{\V p},|\cdot|_{\V \phi}) \subset \Spc S'$ and
$\Spc H'_0=(\Spc N_{\V \phi},|\cdot|_{\V p}) \subset \Spc S$ equipped with the (semi-)inner
products
\begin{align}
\langle f_1,f_2\rangle_{\V \phi} = \sum_{n=1}^{N_0} \langle \phi_n,f_1 \rangle\langle \phi_n,f_2 \rangle\\ 
\langle \varphi_1,\varphi_2\rangle_{\V p} = \sum_{n=1}^{N_0} \langle p_n,\varphi_1 \rangle \langle p_n,\varphi_2 \rangle,
\end{align}
which are well-defined for any $f_1,f_2 \in \Spc S'$ and $\varphi_1,\varphi_2 \in \Spc S$.
The argument generalizes to $\Spc N_{\V p}\subset \Spc H$, and $\Spc N_{\V \phi}\subset \Spc H'$ where $\Spc H$ is any Hilbert space such that $\Spc S \embedC \Spc H \embedC \Spc S'$, which  enables the use of {\em non-universal} biorthogonal systems.
\end{theorem}
\begin{proof}[\small Sketch of proof]
{\small 
The critical part is the existence of a biorthogonal system (universal or not), which follows from the Hahn-Banach theorem. The other properties can be verified by simple substitution, with the inner product between two generic elements of $\Spc N_{\V p}$ being
\vspace*{-3ex}
\begin{align}
\big\langle \sum_{n=1}^{N_0}a_n p_n,\sum_{m=1}^{N_0}b_m p_m\big\rangle_{\V \phi}= \sum_{n=1}^{N_0}a_n b_n=\M a^\Tr\M b,
\end{align}
 which also shows that $\Spc H_0$ is isomorphic to 
$\R^{N_0}$.
} 
\end{proof}

The Hilbert spaces $\Spc H_0$ and $\Spc H_0'$ in Theorem \ref{Theo:FiniteDimHilbert} are in isometric correspondence with the relevant Riesz maps being
\begin{align}
\Op J_{\Spc H_0}=\sum_{n=1}^{N_0} \phi_n \langle \phi_n,\cdot \rangle=\Op R_{\V \phi}: \Spc H_0 \to \Spc H'_0\\\
\Op J_{\Spc H'_0}=\sum_{n=1}^{N_0} p_n \langle p_n,\cdot \rangle=\Op R_{\V p}: \Spc H'_0 \to \Spc H_0.
\end{align}
We also note that the biorthogonality condition 
 is equivalent to $\Op R_{\V p}\Op R_{\V \phi}=\Identity \mbox{ on } \Spc N_{\V p}$ (resp., $\Op R_{\V \phi}\Op R_{\V p}=\Identity \mbox{ on } \Spc N_{\V \phi}$), which leads to the specification of the projection operators
\begin{align}
{\rm Proj}_{\Spc N_\V p}=\Op R_{\V p}\Op R_{\V \phi}=\sum_{n=1}^{N_0} p_n \langle \phi_n,\cdot \rangle: \Spc H \to \Spc N_{\V p}\\
{\rm Proj}_{\Spc N_\V \phi}=\Op R_{\V \phi}\Op R_{\V p}=\sum_{n=1}^{N_0} \phi_n \langle p_n,\cdot \rangle: \Spc H' \to \Spc N_{\V \phi},
%
\end{align}
with the latter being the transpose of the former.

To construct the Hilbert space $ \Spc H \embedC \Spc S'$ (or rather its dual $\Spc H'$), we shall rely
on a special type of operator that will then be extended to yield a Riesz map.
\begin{definition}[Positive-definite operator]
\label{Def:CondPositiveDef2}
Let  $\Op A$ be a continuous operator $\Spc S \to \Spc S'$ and $\Spc N_{\V p}$ a finite-dimensional subspace of $\Spc S'$ that is spanned by
$\V p=(p_1,\dots,p_{N_0})$.
The operator $\Op A$ is said to be
\begin{itemize}
\item symmetric or self-adjoint if, for all $\varphi_1,\varphi_2 \in \Spc S$,
$
\langle \Op A \varphi_1,\varphi_2\rangle=\langle \Op A \varphi_2,\varphi_1\rangle;
$
\item positive-definite if, for any  $\varphi \in \Spc S\backslash\{0\}
$,
$
\langle \Op A \varphi,\varphi\rangle > 0$;

\item positive-semi-definite if, for any  $\varphi \in \Spc S$,
$
\langle \Op A \varphi,\varphi\rangle \ge 0;
$\item $\V p$-conditionally positive-definite if
\begin{align}
\label{Eq:PositiveDefinite}
\langle \Op A \varphi,\varphi\rangle > 0
\mbox{ for all }\varphi \in  \Spc S_{\V p} \backslash\{0\},
\end{align}
where $\Spc S_{\V p}=\Spc S \cap \Spc N_{\V p}^\perp=\{\varphi \in \Spc S: \V p(\varphi)=\V 0\}$. 
\end{itemize}
\end{definition}
Given a biorthognal system $(\V p,\V \phi)$ as in Theorem \ref{Theo:FiniteDimHilbert}, we now introduce a corrected form of the ``Gram'' operator 
\begin{align}
\label{Eq:Apadmis}
\Op A_{\V \phi}&=(\Identity-\Op R_{\V p}\Op R_{\V \phi})\,\Op A\,(\Identity-\Op R_{\V \phi}\Op R_{\V p}): \Spc S \to \Spc S'
\end{align}
that is equivalent to $\Op A$ on $\Spc S_{\V p}$, but more convenient mathematically because
it annihilates every member of $\Spc N_{\V \phi}$.

\begin{theorem}[Operator-based construction of Hilbert spaces]
\label{Theo:conditionPos}
Let $\Op A: \Spc S \to \Spc S'$ be a $\V p$-conditionally positive-definite operator.
Then, for any given biorthogonal system $(\V p, \V \phi)$, the bilinear form
\begin{align}
\label{Eq:BPinnerA}
\langle f,g\rangle_{\Spc H'}=\langle(\Op A_{\V \phi} + \Op R_{\V p}) f,g\rangle
\end{align}
is a valid inner product on $\Spc S$. By completion, the latter specifies the Hilbert space $\Spc H'=\overline{(\Spc S,\|\cdot\|_{\Spc H'})}=\Spc H_{\Op A} \oplus \Spc N_{\V \phi}$, with
$\Spc H_{\Op A}=\{f \in \Spc H': \V p(f)=\V 0\}=\overline{(\Spc S_{\V p},\|\cdot\|_{\Spc H'})}$ the Hilbert space that is the orthogonal complement of $ \Spc N_{\V \phi}$ in $\Spc H'$. 

The predual of $\Spc H'$
is the Hilbert space $\Spc H=\Spc H'_{\Op A} \oplus \Spc N_{\V p}$ associated with the inner product $\langle f,g\rangle_{\Spc H}=\langle(\Op A^{-1} + \Op R_{\V \phi}) f,g\rangle$,
where $\Op A^{-1}$ is the unique operator such that $\Op A^{-1}|_{\Spc H'_{\Op A}}=\Op A_{\V \phi}^{-1}=\Op J_{\Spc H'_{\Op A}}: \Spc H'_{\Op A} \to \Spc H_{\Op A}$ (inverse of the Riesz map $\Op J_{\Spc H_{\Op A}}$) and $\Op A^{-1}|_{\Spc N_{\V p}}=0$ so that its null space is precisely $\Spc N_{\V p}$.

\end{theorem}
\begin{proof}[\small Proof]
{\small
We first select a universal system $\V \phi$ with $\phi_n \in \Spc S$, which ensures the continuity of
$\Op R_{\V \phi}: \Spc S' \to \Spc N_{\V \phi}$. Next, we define the seminorms $|\varphi|_{\Op A}=\sqrt{\langle \Op A_{\V \phi}\varphi,\varphi\rangle}$ and $|\varphi |_{\V p}=\sqrt{\langle \Op R_{\V p}\varphi,\varphi\rangle}$, with the latter reverting to the norm of the finite-dimensional Hilbert space $\Spc N_{\V \phi}$ for any $\varphi \in \Spc N_{\V \phi}$.
We then consider the direct-sum decomposition
$\Spc S=\Spc S_{\V p} \oplus \Spc N_{\V \phi}$ and observe that $|\varphi|_{\Op A}$ specifies a valid norm on
$\Spc S_{\V p}$ (from the definition of conditional positivity), while $|\phi|_{\Op A}$ vanishes for all $\phi \in \Spc N_{\V \phi}$ due to the reproduction condition $\Op R_{\V \phi}\Op R_\V p=\Identity$ on $\Spc N_{\V \phi}$. As for $|\varphi|_{\V p}$, we have exactly the reverse behavior with $|\varphi|_{\V p}=0$ for all $\varphi \in \Spc S_{\V p}$. These properties imply that
the augmented operator $\Op A_{\V \phi} + \Op R_{\V p}$ is positive-definite, with the induced norm
$\|\varphi\|_{\Spc H'}=\sqrt{|\varphi|^2_{\Op A}+|\varphi|^2_{\V p}}$ being compatible with the direct-sum decomposition.
This then allows us to specify the corresponding Hilbert spaces $\Spc H'=\overline{(\Spc S,\|\cdot\|_{\Spc H'})}$ and $\Spc H_{\Op A}=\overline{(\Spc S_{\V p},\|\cdot\|_{\Spc H'})}=\overline{(\Spc S_{\V p},|\cdot|_{\Op A})}$ by completion, with the direct-sum structure $\Spc H'=\Spc H_{\Op A} \oplus \Spc N_{\phi}$ being preserved in the process. 

An important point is that the (semi-)norm $|f|_{\Op A}$ on $\Spc H_{\Op A}=\{f \in \Spc H': \V p(f)=\V 0\}$ is not dependent on the choice of biorthogonal system, in the sense that
\begin{align}
\forall f \in \Spc H_{\Op A}: |f|^2_{\Op A}=\langle \Op A_{\V \phi} f, f\rangle=  \langle \Op A f, f\rangle \mbox{ for any } \V \phi \subset \Spc H' \mbox{ s.t. } \Op R_{\V \phi}\Op R_{\V p}=\Identity.
\end{align}
This allows us to extend the result to general biorthogonal systems subject to the constraint $\phi_1, \dots, \phi_{N_0} \in \Spc H'$, the bottom line being that these systems are all topologically equivalent,
with the underlying space $\Spc H'$ (as a set) remaining the same.

The second part of the theorem is obtained by duality, based on the property that 
$\Spc H=(\Spc H')'=(\Spc H_\Op A \oplus \Spc N_{\V \phi})'=\Spc H_\Op A' \oplus \Spc N'_{\V \phi}$. There, $\Spc N'_{\V \phi}=\Spc N_{\V p}$ (see Theorem \ref{Theo:FiniteDimHilbert}) and $\Spc H_\Op A'$ is a bona fide Hilbert space that is orthogonal to $\Spc N'_{\V \phi}$ (and, hence, to $\V p$) and whose Riesz map  $\Op J_{\Spc H'_\Op A}$ is the inverse of $\Op J_{\Spc H_\Op A}=\Op A: \Spc H_\Op A \to \Spc H'_\Op A$.
} 
\end{proof}
Let us note that \eqref{Eq:BPinnerA} also directly yields the Riesz map $\Op J_{\Spc H'}=\Op A_{\V \phi} + \Op R_{\V p}:
\Spc H' \to \Spc H$, with the latter being the ``native'' space of our interest.

\section{Deterministic Theory: Generalized Splines}
\label{Sec:Deterministic}
The foundation of this theory is a concrete specification of native spaces based on an intermediate (quasi-invertible) operator $\Lop: \Spc H \to L_2$.
This then enables us to formulate a general representer theorem that yields the explicit parametric form of generalized splines.
As it turns out, these splines span the solution spaces of a class of inverse problems much broader than \eqref{Eq:GsplineInt}. We shall then briefly demonstrate how classic regularization theory, reproducing-kernel Hilbert spaces, and the kernel methods of machine learning can be retrieved as special cases of the proposed framework.
\subsection{Native Space Associated to an Operator $\Lop$}
\label{Sec:NativeSpaces}
\begin{definition}[Admissibility]
\label{Def:pAdmissibility}
An operator $\Lop: \Spc S \to L_2$ is said to be $\V p$-admissible if it satisfies the following conditions:
\begin{enumerate}
\item Quasi-invertibility: There exists an inverse operator $\Lop^{-1}$ such that $\Lop^{-1}\Lop \varphi=\Lop\Lop^{-1}\varphi=\varphi$ for all $\varphi \in \Spc S$.
\item Its Gram operator $\Op A=(\Lop^\ast\Lop)^{-1}: \Spc S \to \Spc S'$ is $\V p$-conditionally positive.
\item Null-space property: $\Spc N_{\V p}\subseteq \Spc N_\Lop\subset \Spc S'$, where $\Spc N_\Lop$ is the null space of $\Lop$ extended to the largest possible class of functions/distributions in $\Spc S'$.
\end{enumerate}
\end{definition}
An operator $\Lop$ that is $\V p$-admissible with $\V p=\{0\}$ (trivial null space) will be called {\em invertible} (as opposed to quasi-invertible)
because it admits an extension $\Lop: \Spc H=\overline{(\Spc S,\|\Lop \cdot\|_{L_2})} \to L_2$ that is an isometric bijection. This enables us to (almost trivially) identify its native space as $\Spc H=\Lop^{-1}(L_2)=\{f \in \Spc S': \|\Lop f\|_{L_2}<\infty\}$, which  simplifies the theory considerably.
Observe that for $\Spc S=\Spc S'=\R^N$ (finite dimensions), the family of $\V p$-admissible operators reduces to the invertible ones.

When the operator $\Lop$ is shift-invariant, its (quasi-)inverse and Gram operators in Definition \ref{Def:pAdmissibility} are given by $\Lop^{-1}: \varphi \mapsto \rho_\Lop \ast \varphi$ and $(\Lop^\ast\Lop)^{-1}: \varphi \mapsto \rho_{\Lop^\ast\Lop} \ast \varphi$, where $\rho_\Lop=\Lop^{-1}\{\delta\}$ and $\rho_{\Lop^\ast\Lop}=(\Lop^\ast\Lop)^{-1}\{\delta\}$
are the Green's function of $\Lop$ and $\Lop^\ast\Lop$, respectively. The canonical example of a $\V p$-admissible operator with a null space formed of polynomials is the $m$th-derivative operator $\Dop^m: \Spc S(\R) \to \Spc S(\R)$ with the relevant components being listed in Table \ref{Tab:Derivatives}.
\begin{table}\label{tab:convolution}

\begin{tabular}{p{0.6cm}p{2.4cm}p{2.6cm}p{0.5cm}p{4cm}}
 \hline\\[-2ex]
    $\Lop$ &$\rho_{\Lop}$ & $\rho_{\Lop^\ast\Lop}$ &$N_0$ & $\{(p_n,\phi_n)\}_{n=1}^{N_0}$\\[1ex]
    \hline\\[1ex]
    $\Dop$ & $\tfrac{1}{2}{\rm sign}(x)$ & $-\tfrac{1}{2}|x|$ & $1$ & $\left\{(p_1(x)=1,\phi_1=\delta\right)\}$ \\[2ex]
     $\Dop^m$ & $\tfrac{1}{2}{\rm sign}(x) \frac{x^{m-1}}{(m-1) !}$ & $\tfrac{(-1)^m}{2}\frac{|x|^{2m-1}}{(2m-1)!}$ & $m$ & $\left\{(\frac{x^{n-1}}{(n-1)!}\,,\delta^{(n-1)})\right\}$ \\   
\\
     $\Dop^\alpha$ & $\tfrac{1}{2}{\rm sign}(x) \frac{x^{\alpha-1}}{\Gamma(\alpha)}$ & $\tfrac{1}{2 \sin(\pi  \tfrac{2\alpha-1}{2})}\frac{|x|^{2\alpha-1}}{\Gamma(2\alpha)}$ & $\lfloor \alpha\rfloor$ & $\left\{(\frac{x^{n-1}}{(n-1)!}\,,\delta^{(n-1)})\right\}$\\[2ex]        \hline
\end{tabular} \caption{Differential operators encountered in spline theory with associated Green's functions
and biorthogonal systems.\label{Tab:Derivatives} }
\end{table}

\begin{theorem}[Factorization and $L_2$ isometries]
\label{Theo:FactoRiesz}
Let $\Spc H'=\Spc H_\Op A \oplus \Spc N_{\V \phi}$
be the Hilbert space in Theorem \ref{Theo:conditionPos}, with
the Gram operator now being given by $\Op A=(\Lop^\ast\Lop)^{-1}$ where $\Lop$ is $\V p$-admissible.
Then,
$\Op A_{\V \phi}=\Op G_{\V \phi}^\ast \Op G_{\V \phi}: \Spc H' \to \Spc H'_{\Op A}$ with $\Op G_{\V \phi}=\Op \Lop^{-1\ast}(\Identity-\Op R_{\V \phi}\Op R_{\V p}): \Spc H' \to L_2$, where the null space of $\Op G_{\V \phi}$ is $\Spc N_\V \phi$ and its restriction $\Op G_{\V \phi}|_{\Spc H_\Op A}=\Lop^{-1\ast}: \Spc H_{\Op A} \to L_2$
is a  bijective isometry.
Likewise, $\Op A^{-1}: \Spc H=\Spc H'_{\Op A} \oplus \Spc N_\V p \to \Spc H_{\Spc A}$ factorizes as $\Op A^{-1}=\Op L^\ast \Op L$ with 
$\Lop|_{\Spc H'_\Op A}= \Op G_\V \phi^{-1\ast}: \Spc H'_{\Op A} \to L_2$ being a bijective isometry and
$\Lop|_{\Spc N_\V p}=0$ (null-space property).  
\end{theorem}
\begin{proof} [\small Proof]
{\small By virtue of the factorization of $\Op A_\V \phi$, we can now specify the primary seminorm as
$|f|_{\Op A}=\|\Op G_\V \phi f\|_{L_2}$ with the null-space property of $\Op G_\V \phi$ following from the observation that $(\Identity-\Op R_{\V \phi}\Op R_{\V p})\phi=0$ for all $\phi \in \Spc N_{\V \phi}$.
This shows that the restriction of $\Op G_\V \phi: \Spc H' \to L_2$ to
$\Spc H_{\Op A}$ is an isometry, a fact that can also be stated as $\Op G_\V \phi|_{\Spc H_\Op A}=\Lop^{-1\ast}: \Spc H_{\Op A} \to L_2$. 

A similar picture emerges for the predual space with $|f|'_{\Op A}=\|\Op L f\|_{L_2}$ so that
$\Lop: \Spc H=\Spc H'_\Op A \oplus \Spc N_{\V p}\to L_2$, with $\|\Lop\|=1$ and an exact preservation of the input norm for $\Lop: \Spc H'_A\to L_2 $ (isometry).
To show that the latter map is invertible (bijection), we identify $\Op G_\V \phi^\ast=(\Identity-\Op R_{\V p}\Op R_{\V \phi})\Op \Lop^{-1}$ and verify that
\begin{align}
\Lop \Op G_\V \phi^\ast \varphi=\Lop\Lop^{-1}\varphi=\varphi  \quad \mbox{ for all }  \varphi \in \Spc S,	
\end{align}
as consequence of the admissibility conditions 1 and 3, with the latter yielding $\Lop(\Op R_{\V p}\Op R_{\V \phi})\varphi=0$. This tells us that the isometric map $\Op G_\V \phi^\ast: (\Spc S, \|\cdot\|_{\Spc H'_{\Op A}}) \to (\Spc S, \|\cdot\|_2)$ is invertible with $\Lop\Op G_\V \phi^\ast=\Identity$. By applying the continuous linear extension (BLT) theorem \cite{Reed1975} twice, we then get that $\Op G_\V \phi^\ast: \overline{(\Spc S, \|\cdot\|_{\Spc H'_{\Op A}})}=\Spc H'_{\Op A} \to L_2$ with $\|\Op G_\V \phi^\ast\|=1$ 
and $\Op L: \overline{(\Spc S, \|\cdot\|_2)}=L_2 \to \Spc H'_{\Op A}$ with $\|\Op L\|=1$, which
proves that $\Spc H'_{\Op A}=\Op G_\V \phi^\ast(L_2)$ and $L_2=\Op L(\Spc H_\Op A')$ are isomorphic. 
This then enables us to identify the corresponding Riesz map as
\begin{align}
\Op J_{\Spc H}=\Lop^\ast \Lop + \Op R_{\V p}: \Spc H \to \Spc H'.
\end{align}
Finally, by duality, we get that $L_2=\Op G_\V \phi(\Spc H_{\Op A})$ and $\Spc H_\Op A=\Lop^\ast(L_2)$.
} 
\end{proof}
Theorem \ref{Theo:FactoRiesz} rigorously specifies the {\em native Hilbert space} $\Spc H=\Spc H_\Lop$ associated with a $\V p$-admissible operator $\Lop$.
The concrete transcription of the result is that any $f \in \Spc H_\Lop$ has a unique decomposition
$f=\Op G_\V \phi^\ast w + p$, 
$w=\Lop f$, $p=\Op R_{\V p}\Op R_{\V \phi}f$ so that
\begin{align}
\label{Eq:NormNative}\|f\|^2_{\Spc H_\Lop}=\|\Lop f\|^2_{L_2}+\sum_{n=1}^{N_0} |\langle \phi_n,f\rangle|^2=\|w\|^2_{L_2}+|p|^2_{\V \phi}.
\end{align}
The admissibility of $\Lop$ ensures that $\Lop^\ast\circ \Lop: \Spc S \to L_2 \to \Spc S'$, which implies that $\Spc S \embedD \Spc H_\Lop \embedD \Spc S'$, the embedding being dense and continuous. This, together with \eqref{Eq:NormNative}, then yields the concise description 
$\Spc H_\Lop=\overline{(\Spc S,\|\cdot\|_{\Spc H_\Lop})}$.

By using the property that $\Spc H_\Lop$ (as a set) is independent of the choice of biorthogonal system $(\V p, \V \phi)$, we can show that the adjoint operator $\Op G^\ast_{\V \phi}:
L_2 \to \Spc H_\Lop$ is the unique right-inverse of $\Lop$ that imposes the boundary condition $\V \phi(f)=\V 0$. To remind us of these properties, we shall for now on denote $\Op G^\ast_{\V \phi}$ by $\Lop^{{-1}}_{\V \phi}$ with the definition of this operator being
\begin{align}
\label{Eq:gInverse}
\Op L^{-1}_{\V \phi}=(\Identity-\Op R_{\V \phi}\Op R_{\V p})\Lop^{-1}: L_2 \to \Spc H_\Lop.
\end{align}
This also suggest an alternative definition of $\Op L^{-1}_{\V \phi}\{w\}$
as the unique solution in $\Spc H_\Lop$ of the linear differential equation
\begin{align}
\Lop f=w\quad \mbox{s.t.}\quad \V \phi(f)=\V 0
\label{Eq:PDE}
\end{align}
with $w\in L_2$. 
\subsection{Generalized Variational Splines}
Our variational splines are specified as solutions of inverse problems where a function (or a distribution) $f \in \Spc H_\Lop$ 
is probed through a set of linear measurement functionals $\V \nu=(\nu_1, \dots,\nu_M)$.
Since the recovery of $f$ from such measurements is generally ill-posed, one seeks the ``most regular'' solution $f_{\rm spline}\in \Spc H_\Lop$ that fits the data, with our notion of 
regularity being tied to an admissible operator $\Lop$.
In the noise-free scenario, this is formulated as $f_{0}=\arg \min_{f \in \Spc H_\Lop} \|\Lop f\|^2_{L_2}$ s.t. $\M y=\V \nu(f)$.
By using a $\V p$-admissible operator, one privileges solutions that are in the null space $\Spc N_{\V p}$ of $\Lop$, which is one of the important specificities of splines.
When the measurements are corrupted by noise, one adapts the methodology by approximating the data within the tolerance of the noise, rather than fitting them exactly. In practice, this is done by reformulating the task as the minimization of a combined loss that is the sum of a data-fidelity term 
and a regularizing ``spline energy.'' 

Our next result shows that the solutions of a broad class of such problems all have the same parametric form $\eqref{Eq:f0lin2}$: a generalized spline that lives in a finite-dimensional reconstruction space whose basis functions
solely depend on $(\Lop, \V \nu)$.
We obtain it as a special case of \cite[Theorem 3]{Unser2022} by exploiting the explicit direct-sum structure of the native space $\Spc H_\Lop$ in Theorem \ref{Theo:conditionPos}.

\begin{theorem}[Representer theorem for generalized variational splines]
\label{Theo:GeneralRepSemiHilbert}
We consider the following setting.
\begin{itemize}
\item A regularization operator $\Lop: \Spc S \to L_2$ that is $\V p=(p_0,\dots,p_{N_0})$-admissible with $\Op A=(\Lop^\ast\Lop)^{-1}: \Spc S \to \Spc S'$.
\item The native space $\Spc H_\Lop=\overline{(\Spc S,\|\cdot\|_{\Spc H_\Lop})}$ equipped with the norm \eqref{Eq:NormNative} and its continuous dual $\Spc H_\Lop'=\Spc H_{\Op A} \oplus \Spc N_{\V \phi}$.
\item A linear measurement operator 
$\V \nu: \Spc H_\Lop \to \R^M$ with $M>N_0$. 

\item A  loss functional $E: \R^M \times \R^M \to \R_{\ge0}\cup \{+\infty\}$ that is lower-semicontinuous, coercive, 
and convex in its second argument.
\item Some fixed regularization parameter $\lambda \in \R_{\ge0}$. 
\end{itemize}
To ensure that \eqref{Eq:GenericOptimizationProb2} is well-posed, we also require that the vectors $\M v_1,\dots,\M v_{N_0} \in\R^M$ with 
$[\M v_n]_m=\langle \nu_m,p_n \rangle$ be {\em linearly independent}, which implies the existence of
a complementary set $\{\M u_1, \dots, \M u_{M-N_0}\}$ in $\R^M$ such that $\R^M={\rm span} \{\M v_n\}_{n=1}^{N_0}  \oplus {\rm span} \{\M u_m\}_{m=1}^{M-N_0}$.

Then, for any fixed $\M y\inR^M$,  the solution of the optimization problem
\begin{align}
\label{Eq:GenericOptimizationProb2}
\min_{f \in \Spc H_\Lop} \left( E\big(\V y, \V \nu(f)\big)+ \lambda \|\Lop f\|^2_{L_2}\right)
\end{align}
is unique and admits the linear parameterization 
\begin{align}
\label{Eq:f0lin2}
f_\lambda=\sum_{m=1}^{M-N_0} a_m \widetilde{\psi}_m +\sum_{n=1}^{N_0} b_{n} p_n,
\end{align}
with coefficients $(\V a,\V b)\inR^{M}$ and 
fixed basis functions $p_n\in \Spc N_\V p$, $\widetilde{\psi}_m=\Op A\{\widetilde{\nu}_m\}\in \Spc H_{\Lop}$, where
$\Op A=(\Lop^\ast\Lop)^{-1}$ is the Riesz map $\Spc H_{\Op A} \to \Spc H'_{\Op A} $ 
and  $\widetilde{\nu}_m=\widetilde{\M u}_m^\Tr\V \nu \in \Spc H_\Op A$. Here, $\widetilde{\M u}_m \in\R^M$ is the unique (biorthogonal) vector such that $\widetilde{\M u}^\Tr_m\M v_n=0$ and $\widetilde{\M u}^\Tr_m\M u_{m'}=\delta_{m,m'}$
for any $m,m' \in \{1,\dots,M-N_0\}$ and $n\in \{1,\dots,N_0\}$.


\end{theorem}

One obtains an exact fit, as in \eqref{Eq:GsplineInt}, either by letting $\lambda\to 0$ or by selecting $E$ to be a barrier function that returns
$0$ if $\V \nu(f)=\M y$ and $+\infty$ otherwise. At the other end of the scale, for $\lambda\to \infty$, one gets a linear regression with $f_\infty \in \Spc N_{\V p}$.

We like to point out that the parametric form \eqref{Eq:f0lin2} in Theorem \ref{Theo:GeneralRepSemiHilbert} differs from the traditional representation of such variational splines. The standard form is\begin{align}
\label{Eq:f0lin3}
f_\lambda=\sum_{m=1}^{M} \tilde a_m \psi_m +\sum_{n=1}^{N_0} \tilde b_{n} p_n,
\end{align}
with $\psi_m=\Op A\{\nu_m\}$, which requires the specification of additional ``orthogonality'' conditions between the expansion coefficients $\tilde{\M a}=(\tilde a_m)$ and $\tilde{\M b}=(\tilde b_{n})$ not given here (see \cite{Wahba1990} for details). While the two forms are equivalent, they each have their specificities.
The conventional form \eqref{Eq:f0lin3} has the advantage of not requiring the choice of any boundary functionals $\V \phi$. The fact that it is overparameterized is dealt with by enforcing the ``orthogonality'' of the two components, which results in an additional set of $N_0$ linear equations. The aspect that is problematic in an abstract theory is that there is no guarantee that
the ``augmented'' basis functions $\psi_m=\Op A\{\nu_m\}$ are well-defined, because, unlike the $\widetilde{\psi}_m$ in Theorem \ref{Theo:GeneralRepSemiHilbert}, they are typically not included in $\Spc H_\Lop$.
Eq.~\eqref{Eq:f0lin2}, by contrast, has the advantage that the linear parametrization of the spline is explicit and foolproof, but it requires more work to construct the reduced basis functions, which can be linked to a specific $\V \phi$. While this may look as an unnecessary complication, we shall discover that it is actually the key for making the connection with a corresponding class of generalized Gaussian processes, which would otherwise be ill-defined.

\subsection{Tikhonov Regularization and Ridge Regression}
\label{Sec:Tikhonov}
Remarkably, the infinite-dimensional optimization problem \eqref{Eq:GenericOptimizationProb2}  can be solved in closed form whenever the loss $E$ is quadratic, and this irrespective of the actual choice of Hilbert spaces. The same holds true for \eqref{Eq:GsplineInt}.
As demonstration of usage, we set $E\big(\M y, \V \nu(f)\big)=\sum_{m=1}^M |y_m -\langle \nu_m, f\rangle|^2$ and assume (for simplicity) that
the null space of the regularization functional is trivial ($N_0=0$). By plugging the solution \eqref{Eq:f0lin2} into \eqref{Eq:GenericOptimizationProb2} and by defining the system matrix
$\M G \in \R^{M \times M}$ with entry $[\M G]_{m,n}=\langle \nu_m,\psi_n \rangle$, we then recast the problem as
\begin{align*}
\min_{\M a \in  \R^M} \left(\|\M y - \M G \M a\|^2 + \lambda \|\sum_{m=1}^M a_m \psi_m\|^2_{\Spc H} \right)
=\min_{\M a \in  \R^M} \left(\|\M y - \M G \M a\|^2 + \lambda \M a^\Tr \M G \M a\right).
\end{align*} 
The key here is Item 3 in Theorem \ref{Def:DualMap}
 with $\nu=\nu_m$ and $f=\psi_n = \nu_n^\ast$,
which gives $\langle \nu_m, \psi_n\rangle=\langle \nu^\ast_m, \psi_n\rangle_{\Spc H}=\langle \psi_m, \psi_n\rangle_{\Spc H}$
and makes  the underlying system matrix equal to the
Gram matrix of the basis $\{\psi_m\}$. This then yields the solution
\begin{align}
\label{Eq:TikAbstract}
f_\lambda= \sum_{m=1}^M a_m \psi_m\quad  \mbox{ with }\quad \M a=(\M G \M G + \lambda \M G)^{-1} \M G \M y=(\M G + \lambda \M I)^{-1} \M y,
\end{align}
where we have made use of the invertibility of $\M G$. The latter, which is equivalent
to the linear independence of the $\psi_m$, is inherited from the linear independence of the $\nu_m$ through the Riesz pairing.

\subsection{Link with RKHS and Conventional Splines}
We can make the link with reproducing-kernel Hilbert spaces when $\Spc H$ is a Hilbert space of continuous functions on $\R^d$
on which the sampling functionals are well-defined, with $\nu_m=\delta_{\V x_m} \in \Spc H'$ for any $\V x_m \in \R^d$.
In that setting with $\Spc N_{\V p}=\{0\}$, the {\em reproducing kernel} $h: \R^d \times \R^d \to \R$ is the Schwartz kernel (a.k.a.\ generalized impulse response) of the symmetric, positive-definite operator $\Op A$ in Theorem \ref{Theo:conditionPos}, whose action is described explicitly by the integral $\Op A\{\varphi\}(\V x)=\int_{\R^d} h(\V x,\V y) \varphi(\V y) \dint \V y$. 

If we now consider the corresponding interpolation/approximation problem \eqref{Eq:GenericOptimizationProb2} with $\langle\nu_m, f\rangle=\langle \delta_{\V x_m}, f\rangle=f(\V x_m)$
at distinct locations $\V x_1, \cdots,\V x_M \in \R^d$, we find that the basis functions in \eqref{Eq:f0lin2} are given by
$\psi_m=\Op A\{\delta_{\V x_m}\}=h(\cdot,\V x_m)$.
This yields $f_\lambda(\V x)=\sum_{m=1}^M a_m h(\V x,\V x_m)$, which has the form of the kernel estimators used in machine learning.

As for the case where $\Op A$ is shift-invariant with frequency response $\widehat a(\bw)$, we note that the abstract definition of positive-definiteness in Definition \ref{Def:PDFunctional} below
with $\varphi_m=\delta_{\V x_m} \in \Spc H'$
reverts to the definition of a positive-definite function (in the sense of Bochner), which is fundamental to the Fourier-based formulation of probability theory. This notion is also central to interpolation theory: a shift-invariant kernel $h(\V x,\V y)=a(\V x- \V y)$ is the reproducing kernel of a Hilbert space if $\widehat a(\bw)$ (the Fourier transform of $a: \R^d\to \R$) is a (strictly) positive function.

We can also make the link with polynomial-spline interpolation by setting $\Spc S=\Spc S(\R)$, $\Lop=\Dop^N$, and 
$\nu_m=\delta_{x_m} \in \Spc H'_{\Dop^N}$. By making use of the biorthogonal system and Green's function listed in Table \ref{Tab:Derivatives}, one readily shows by regrouping terms that \eqref{Eq:f0lin2} then takes the form
\begin{align}
f_\lambda(x)=\sum_{1=1}^M \tilde{a}_m |x-x_m|^{2N-1} +
 \sum_{n=1}^{N-1} \tilde{b}_n x^{n-1},
\end{align}
which is the classic formula of a polynomial spline of degree $(2N-1)$ with knots at the $x_m$.

\section{Stochastic Theory: The Gaussian Connection}
\label{Sec:GGP}

\subsection{Abstract Generalized Gaussian Processes}
The definition of abstract generalized Gaussian processes is guided by two important observations: the first is that any linear functional of a Gaussian
random vector is a scalar Gaussian random variable. The second is that the characteristic function of the Gaussian random vector $\VRand X\sim\Spc N(\V \mu_{\VRand X},\M C_{\VRand X})$ with mean vector $\V \mu_{\VRand X}\in \R^N$ and covariance matrix  $\M C_{\VRand X}\in \R^{N \times N}$ is given 
by 
\begin{align}
\label{Eq:MultiGaussFourier}
\hat p_{\VRand X}(\V \xi)= \Exp\{\ee^{\jj \langle \V \xi, \M x\rangle}\}=\exp\left( \jj \langle \V \mu_{\VRand X},\V \xi\rangle - \tfrac{1}{2} \langle \M C_{\VRand X}\V \xi,\V \xi\rangle \right).
\end{align}
We recall that the latter is the (conjugate) Fourier transform of the probability density function (pdf) 
$p_{\VRand X}$ of the random vector $\VRand X$. 

Gaussian stochastic processes, where the objects of interest are random functions, are the classic infinite-dimensional extensions of Gaussian random variables.
To fully exploit the mathematical framework of Section 1.2, we shall take one more step of abstraction by considering generalized Gaussian processes, which will be defined as abstract random linear functionals acting on some nuclear space $\Spc S$. 
The proposed framework is general enough to cover all Gaussian random objects, including scalar variables (with $\Spc S=\R$), vectors (with $\Spc S=\R^N$), discrete signals with $\Spc S=\Spc S(\Z)$, and (multidimensional) generalized functions with $\Spc S=\Spc S(\R^d)$ or $\Spc S=\Spc D(\Omega)$ with $\Omega \subseteq \R^d$. 

To explain the concept, we recall that an element
$g\in \Spc S'$ (the dual space of $\Spc S$) is a continuous linear functional $g: \Spc S \to \R$ that maps $\varphi \mapsto g(\varphi)=\langle g, \varphi\rangle \in \R$. We can therefore view a generalized Gaussian process $\gRand g$ as a random generation mechanism that produces such $g$ (the so-called {\em realizations} of $\gRand g$) so that $\gRand g(\varphi)=\langle G, \varphi\rangle$ with $\varphi$ fixed yields a scalar Gaussian random variable with a predictable mean and variance. Such a generalized process can be characterized mathematically by a Gaussian probability measure on $\Spc S'$ or, equivalently, by its characteristic functional on $\Spc S$---the infinite-dimensional counterpart of \eqref{Eq:MultiGaussFourier}.
\begin{definition}
\label{Theo:GGauss}
A {\em generalized Gaussian process} $\{\gRand g(\varphi)=\langle \gRand g,\varphi \rangle : \varphi \in \Spc S \}$ in $\Spc S'$ is a collection of real-valued random variables indexed by $\Spc S$ that has the following properties.
\begin{enumerate}
    \item Linearity: For any \( \varphi_1, \varphi_2 \in \Spc S \) and scalars \( a_1, a_2 \in \R \),
\begin{align*}
    \gRand g(a_1 \varphi_1 + a_2 \varphi_2) = a_1 \gRand g(\varphi_1) + a_2 \gRand g(\varphi_2).
\end{align*}
\item Continuity: For any converging sequence $(\varphi_n)_{n \in \N}$ in $\Spc S$, we have that
\begin{align*}
    \lim_{n \to \infty}\gRand G(\varphi_n) =\gRand G(\lim_{n \to \infty}\varphi_n).
\end{align*}

    \item Gaussianity:
    For each finite set \( \{\varphi_1, \dots, \varphi_N\} \subset \Spc S \), the random vector $\VRand Y=\big(\gRand g(\varphi_1), \dots, \gRand g(\varphi_N)\big)$ has a multivariate Gaussian distribution with the underlying finite-dimensional distributions all being mutually compatible.
\end{enumerate}
\end{definition}

As in the finite-dimensional setting, a generalized Gaussian process (GGP) is characterized by its first and second-order moments, with the difference that the mean $\mu=\Exp\{ \gRand g\}\in \Spc S'$
is now a continuous linear functional on $\Spc S$, while the ``covariance matrix'' takes the form of the symmetric positive-definite operator $\Op A: \Spc S \to \Spc S'$. These enable the determination
of the means and covariances of $\gRand g$ according to the rules
\begin{align*}
\mu(\varphi)&=\Exp\{\gRand g(\varphi)\}=\langle \mu,\varphi\rangle \\
{\rm Cov}_{\gRand G}(\varphi_1,\varphi_2\big)&=\Exp\Big\{\big(\gRand g(\varphi_1)-\mu(\varphi_1)\big)\big(\gRand g(\varphi_2)-\mu(\varphi_2)\big)\Big\}=\langle \Op A \varphi_1,\varphi_2 \rangle.
\end{align*}
The corresponding generalized Gaussian process is denoted $\gRand g \sim \Spc N(\mu,\Op A)$, in direct analogy with the finite-dimensional convention. 

By adapting Gelfand's proof for $\Spc S=\Spc D(\R^d)$ \cite{Gelfand-Villenkin1964}, one can guarantee the existence of such a generalized Gaussian process for any $\mu \in \Spc S'$ and symmetric positive-definite operator $\Op A$. In particular, the generalized Gaussian process $\gRand W \sim \Spc N(0,\Identity)$ is a called a {\em white noise} or {\em Gaussian innovation}.
\begin{theorem}[Generalized Gaussian process in $\Spc S'$]
\label{Theo:GGaussProcess}
For any given $\mu \in \Spc S'$ and
symmetric positive-definite covariance operator $\Op A: \Spc S \to \Spc S'$,
the generalized Gaussian process $\gRand g \sim \Spc N(\mu,\Op A)$ is a random continuous linear functional on $\Spc S$. It is fully specified by its characteristic functional
\begin{align}
\label{Eq:CharacGauss}
\Form_{\gRand g}(\varphi)=
\Exp\Big\{\ee^{ \jj \gRand g(\varphi)}\Big\} = \exp\left( \jj \langle \mu,\varphi\rangle - \tfrac{1}{2} \langle \Op A \varphi, \varphi\rangle \right).
\end{align}
This process induces mutually compatible families of random Gaussian vectors $\VRand Y={\big(\gRand g(\varphi_n), \dots,\gRand g(\varphi_N)\big)} \sim \Spc N(\V \mu_{\VRand Y},\M C_{\VRand Y})$ whose mean vectors $\V \mu_{\VRand Y}\in \R^N$ and covariance matrices $\M C_{\VRand Y} \in \R^{N \times N}$ are given by
\begin{align}
[\V \mu_{\VRand Y}]_n=\langle \mu,\varphi_n\rangle\\
[\M C_{\VRand Y}]_{m,n}={\rm Cov}_{\gRand g}(\varphi_m,\varphi_n)=\langle \Op A \varphi_m,\varphi_n \rangle,
\end{align}
with the symmetric matrices $\M C_{\VRand Y}$ being positive-definite for any combination of linearly independent $\varphi_1,\dots,\varphi_N \in \Spc S$.
\end{theorem}

\begin{proof}[\small Sketch of proof]
{\small
The proof relies on the powerful Minlos-Bochner theorem that applies to nuclear spaces exclusively. It states that a functional $\Form: \Spc S \to \C$ is the characteristic function of a measure on the dual space $\Spc S'$ if and only if it is continuous, positive-definite (see Definition \ref{Def:PDFunctional} below), and such that $\Form(0)=1$. One can then show that these properties are satisfied by $\Form_{\gRand g}$ given by \eqref{Eq:CharacGauss} under the condition that the operator $\Op A$ is positive-definite. 

Once the characteristic functional is known, we determine the finite-dimensional
characteristic function of the vector variable $\VRand Y$ by making the substitution $\varphi=
\wRand_1 \varphi_1+\dots+\wRand_N \varphi_N$ with $\bwRand=(\wRand_n) \in \R^N$ being our Fourier-domain variable. This yields 
\begin{align*}
\hat p_{\VRand y}(\bwRand)&=\Form_{\gRand g}(\wRand_1 \varphi_1+\dots+\wRand_N \varphi_N)\\
&=\exp\left(\jj \sum_{n=1}^N\wRand_n \langle \mu_{\gRand g}, \varphi_n\rangle-\tfrac{1}{2} \sum_{m=1}^N\sum_{n=1}^N\wRand_m\wRand_n\langle\Op A \varphi_m,\varphi_n\rangle\right)\\
&=\exp\left(\jj  \V \mu_{\VRand Y}^\Tr\bwRand-\tfrac{1}{2} \bwRand^\Tr \M C_{\VRand Y} \bwRand\right),
\end{align*}
where we have made use of the (bi)linearity of the duality product and identified
the underlying mean vector $\V \mu_{\VRand Y}$ and covariance  matrix $\M C_{\VRand Y}$.
This is the same as \eqref{Eq:MultiGaussFourier}, which proves that $\VRand y\sim\Spc N(\V \mu_{\VRand Y},\M C_{\VRand Y})$.
} 
\end{proof}
\begin{definition}[Positive-definite functional]  
\label{Def:PDFunctional}
Let $\Spc X$ be a topological vector space (e.g., $\Spc S$ or $\Spc H'$).  
A functional $\widehat{\mathscr{P}} : \Spc X \to \C$ is said to be \emph{positive-definite} if, for all finite collections $\varphi_1, \dots, \varphi_N \in \Spc X$ and $c_1, \dots, c_N \in \mathbb{C}$, we have that\begin{align}
\label{Eq:PSfunction}
\sum_{n=1}^N \sum_{m=1}^N c_n \overline{c_m} \, \widehat{\mathscr{P}}(\varphi_n - \varphi_m) \geq 0.
\end{align}
\end{definition}
For consistency with Definition \ref{Def:CondPositiveDef2}, the Bochner condition \eqref{Eq:PSfunction} of probabilists ought rather be called ``positive semi-definiteness'' because the inequality is not strict. This means that the framework can also accommodate $\V p$-admissible covariance operators of the form given by  \eqref{Eq:Apadmis}.
This, in turns, yields a GGP such that $\gRand g(\phi)=0$ for all $\phi \in \Spc N_{\V \phi}$, which partially spoils the unconditional positive-definiteness of $\M C_{\VRand Y}$ stated in Theorem \ref{Theo:GGaussProcess}.

\subsection{Extended Generalized Gaussian Processes}
\label{Sec:ExtendedProcesses}
We now reveal the remarkable link between the generalized Gaussian processes of Theorem \ref{Theo:GGaussProcess} and the Hilbert spaces of Sections \ref{Sec:NtoHilbert} and \ref{Sec:NativeSpaces}.
To that end, we rewrite their characteristic functional as
$\Form_{\gRand g}(\varphi)= \exp\left( \jj \langle \mu,\varphi\rangle - \frac{1}{2} \|\varphi\|^2_{\Spc H'} \right)$, where
$\Spc H'$ is the (dual) Hilbert space identified in Theorem \ref{Theo:conditionPos} for the simpler case $(\V p=\{0\})$, and where $\Op A$ is symmetric positive-definite.
Even though the corresponding generalized Gaussian process $\gRand g \sim \Spc N(\mu,\Op A)$ is specified to act on test functions $\varphi \in \Spc S$, this allows one to extend its action to any member $\nu \in \Spc H'=\overline{(\Spc S,\|\cdot\|_{\Spc H'})}$ by essentially applying the same continuous extension principle as in the proof of Theorem \ref{Theo:conditionPos}.
This is legitimate, provided that the extended characteristic functional $\Form_{\gRand g}: \Spc H' \to \C$ induces the three defining properties of characteristic functions (Bochner's theorem): normalization with $\Form_{\gRand g}(0)=1$; continuity; and positive-definiteness.
The normalization is obviously conserved, while the continuity on $\Spc H'$ directly follows from the continuity of the underlying Hilbertian norm.
Similarly, one can use a limit/density argument to transfer the positive-definiteness on $\Spc S$ to the larger space $\Spc H'$.

%

\subsection{Innovation Models and Gaussian Solutions of SDEs}
\label{Sec:GaussianSDE}
We shall now use the functional tools of Section \ref{Sec:Deterministic} to identify an important class of GGP.
We first consider a proper GGP in $\Spc S'$ whose positive-definite covariance operator is factorizable as $\Op A=\Op T\Op T^\ast$, with $\Op T: L_2 \to  \Spc H'\embedC \Spc S'$.
The positive-definiteness of $(\Op T\Op T^\ast): \Spc S\to \Spc S'$ implies that each factor is invertible
with the inverse operator $\Lop=\Op T^{-1}$ being admissible in the sense of Definition \ref{Def:pAdmissibility} with $\V p=\{0\}$.
We claim that such a process can be synthesized by applying the linear transformation  $\Op T=\Lop^{-1}$ to the Gaussian white noise (or innovation process) $\gRand w\sim \Spc N(0,\Identity)$. In our formalism, this generation mechanism is expressed as
%
\begin{align}
\label{Eq:GenerateG}
\gRand G(\varphi)=\Op T\{\gRand W\}(\varphi)=\langle\Op T\{\gRand W\},\varphi\rangle=
\langle\gRand W,\Op T^\ast\{\varphi\}\rangle=\gRand w(\Op T^\ast\varphi),
\end{align}
where $\Op T^\ast: \Spc S \to L_2$ is the adjoint of $\Op T: L_2 \to \Spc S'$.
The continuity of $\Op T$ together with the property that the domain of the white noise $\gRand w$ is extendable to $L_2=(L_2)'$ (see Section \ref{Sec:ExtendedProcesses}) ensure that the right-hand side of \eqref{Eq:GenerateG} is well-defined for any $\varphi \in \Spc S$.
This leads to the conclusion that
\begin{align*}
\Form_{\gRand G}(\varphi)=\Form_{\gRand w}(\Op T^\ast \varphi) 
=\exp\left(-\tfrac{1}{2} \|\Op T^\ast\varphi\|_{L_2}^2\right)=\exp\left(-\tfrac{1}{2} \langle(\Op T \Op T^\ast) \varphi, \varphi\rangle\right),
\end{align*}
which proves that $\gRand G \sim \Spc N(0, \Op A)$.
We also observe that the factorization $\Op A=\Op T\Op T^\ast$ is not unique: there is a large equivalence class of possible transformations given by
$\tilde{\Op T}=\Op T\Op U $, where $\Op U$ is an arbitrary unitary operator $L_2 \to L_2$ with the property that $\Op U^{-1}=\Op U^\ast$. 


Next, we move to the scenario where the ``shaping'' operator $\Op T=\Op L^{-1}_{\V \phi}$ is the distributional extension of the operator specified by 
\eqref{Eq:gInverse} (right-inverse of a $\V p$-admissible
 $\Lop$). While this yields a generalized Gaussian process
 $\gRand G=\Op L^{-1}_{\V \phi}\{\gRand W\}$ whose covariance operator is only conditionally positive-definite, it is still well-defined and has a proper IR. Its innovation $\gRand W$ is recovered by applying the {\em whitening operator}
$\Lop$ as in
\begin{align}
\label{Eq:Innovationmodel0}
\Lop\{\gRand G\}: \varphi \mapsto \langle\Op \Lop\{\gRand G\},\varphi\rangle=
\langle\Op \Lop\Op L^{-1}_{\V \phi}\{\gRand W\},\varphi \rangle=\gRand W(\Op L^{{-1}\ast}_{\V \phi}\, \Lop^\ast\varphi)=W(\varphi),
\end{align}
where the adjoint manipulation is legitimate because: (i) the adjoint operator $\Lop^\ast$ continuously maps $\Spc H'_\Lop\to (L_2)'=L_2$ (the extended domain of $\gRand W$); and (ii)
$(\Lop\Op L^{-1}_{\V \phi} )^\ast=\Op L^{{-1}\ast}_{\V \phi}\Lop^\ast=\Identity$ on $L_2$ (by Theorem \ref{Theo:FactoRiesz} with $\Op G_\V \phi=\Op L^{{-1}\ast}_{\V \phi}$).
These properties imply that the process $\gRand g$, or rather its realizations $g \in \Spc S'$, can also be specified as 
solutions of a linear stochastic differential equation (SDE) driven by a Gaussian white noise $w \in \Spc S'$,
which can also be written as \eqref{Eq:PDE}, albeit with a ``weak'' (or distributional) interpretation of the equality.
The intuition there is that the ambiguity induced by the nontrivial null space of $\Lop$ gets resolved by imposing $N_0$ linear boundary conditions.

While the formal description $\gRand G=\Op L^{-1}_{\V \phi}\{\gRand W\}$ of a GGP looks rather innocuous, it is a remarkably powerful alternative to the classic integral representation of Gaussian stochastic processes. We shall make our point by considering a
famous instance of such a representation: Mandelbrot's definition \cite{Mandelbrot1968} of fractional Brownian motion (fBm) with Hurst index $H \in\big[\tfrac{1}{2},\tfrac{1}{2}+1 \big)$ as 
\begin{align}
\label{Eq:fBMH}
\Rand B_H(t)=\int_{\R }\frac{1}{\Gamma(H+\frac{1}{2})} \left(t-\tau)^{H-\frac{1}{2}}_+-(-\tau)_+^{\gamma-1}\right)\dint \Rand B(\tau),
\end{align}
where $\Rand B(\cdot)$ is the Wiener process (considered as a random measure) and $\dint\Rand B$ the representation of Gaussian white noise in the classic notation of stochastic calculus.
The equivalent description of fBm in the present framework is
$B_H=\Op L^{-1}_{\V \phi}\{\Rand W\}$ with $\Lop= \Dop^{\gamma}$, $\gamma=H+\frac{1}{2}  \in [1,2)$, $N_0=1$, and $(p_1,\phi_1)=(1,\delta)$. The operator $\Dop^{\gamma}$
is the fractional derivative of order $\gamma$ with Fourier symbol $(\jj \omega)^\gamma$, which induces fractional splines \cite{Unser2000}. Its Green's function can be found in Table 1. 
One readily verifies that the integral kernel in \eqref{Eq:fBMH} is the Schwartz kernel of the fractional integral operator $\Op L^{-1}_{\V \phi}=\Op D^{-\gamma}_{\delta}$. This equivalence also yields a simpler description of fBm as the solution of the fractional SDE
\begin{align}
\Dop^{H +\frac{1}{2}}\{\Rand B_H\}=\Rand W \quad  \mbox{ s.t.} \quad B_H(0)=0,\end{align}
which provides immediate insights on the fractal and regularity properties of these processes \cite{Blu2007a}.
As expected, the correlation function of fBm coincides with the Schwartz kernel of the covariance operator: 
\begin{align}
 \label{Eq:CovfBm}
 \Exp\{\Rand B_H(t)\Rand B_H(\tau) \}=\Op A_{\V \phi}\{\delta_\tau\}(t)=\tfrac{1}{2}\big(|t|^{2H}+|\tau|^{2H}-|t-\tau|^{2H} \big) \end{align}
with the case of Brownian motion being recovered for $H=\tfrac{1}{2}$ with $\Lop=\Dop$.
\subsection{Generalized Splines as MMSE Estimators of GGP}
\label{Sec:MMSESplines}
We now have all the elements to 
make the connection between the minimum-error estimation of a signal under the Gaussian hypothesis---also known as the {\em Wiener estimator}---and the spline reconstruction techniques investigated in Section \ref{Sec:Deterministic}.
To that end, we consider the linear measurement model
\begin{align}
\label{Eq:Wienerproblem}
g \mapsto \M y=\V \nu(g) + \V \epsilon \in \R^M,
\end{align}
where $g$ (our signal) is a realization of a Gaussian process $\gRand g \sim \Spc N(0,\Op A)$ in $\Spc S'$  
 and where $\V \nu=(\nu_m)$ with $\nu_1,\dots, \nu_M\in \Spc H'$ is a continuous linear measurement operator $\V \nu: \Spc H \to \R^M$.
The second component of the model, $\V \epsilon \in \R^M$, is an additive disturbance term (discrete measurement noise) that does not depend on the signal and whose components are i.i.d.\ Gaussian with zero mean and variance 
$\sigma^2$. 

To keep the argument simple, we focus on the scenario $\Op A=(\Lop^\ast\Lop)^{-1}: \Spc S \to \Spc S'$
with $\Lop$ invertible ($N_0=0$), for which we have explicitly solved
the deterministic inverse problem in Section \ref{Sec:Tikhonov}. Specifically, we know that the {\em Tikhonov estimator} of $g$ in \eqref{Eq:Wienerproblem} given $\M y$ is the generalized smoothing spline
\begin{align}
g_\lambda&=\arg \min_{g \in \Spc H}\left( \|\M y - \V \nu(g)\|^2_2 + \lambda \|\Lop g\|^2_{L_2}\right) \nonumber\\
&=\sum_{m=1}^M a_m \Op A\{\nu_m\}\quad \mbox{ with } \quad (a_m)=(\M G + \lambda \M I)^{-1} \M y, 
\label{Eq:GSmoothingSpline}
\end{align}
where the matrix $\M G \in \R^{M \times M}$ with $[\M G]_{m,n}=\langle \Op A \nu_m, \nu_n\rangle=\langle \nu_m,\nu_n\rangle_{\Spc H'}$ is symmetric positive-definite. 

As alternative, we shall now 
apply the Bayesian paradigm to obtain the MMSE estimator of the current realization $g: \varphi \to \langle g, \varphi\rangle$ of the generalized Gaussian process $\gRand g$ given its measurements $\M y$. The first enabling element is the probabilistic transcription of the linear measurement model 
\eqref{Eq:Wienerproblem} as $\VRand Y=\VRand Y_0+ \V \epsilon$ with $\V \epsilon \sim \Spc N(\V 0, \sigma^2 \M I)$ where
$\VRand Y_0=(\gRand g(\nu_1),\dots,\gRand g(\nu_M))$ is the random noise-free measurement vector induced by $\gRand G$.
The second important insight is that, if we fix the test function $\varphi\in \Spc H'$, then the quantity to estimate,  $\Rand X=\gRand g(\varphi)$,
is a scalar random variable. 

Our strategy to infer $\Rand X$ from $\VRand Y=\M y$ is based on the determination of
the pdf of the augmented vector-variable $\VRand z=(\Rand X, \VRand Y)$, which contains all the necessary statistical information and can be predicted to be Gaussian. As first step, we apply Theorem \ref{Theo:GGaussProcess}
 to the noise-free variable  $\VRand z_0=(\Rand X, \VRand Y_0)$, which gives
$\VRand z_0\sim \Spc N(\V 0, \M C_{\VRand z_0})$ with
\begin{align}
\label{Eq:AugmentedCovar0}
\M C_{\VRand z_0}=\left(\begin{array}{cc}
\sigma_\Rand x^2  &   \M c^\Tr    \\[0.5ex]
\M c  &   \M G 
\end{array}
\right),
\end{align}
where $\sigma_\Rand x^2=\langle \Op A \varphi, \varphi\rangle$, $\M c \in \R^M$ with $[\M c]_m=\langle \Op A \varphi, \nu_m,\rangle=\Exp\big\{\Rand x\hspace{.08em}\Rand y_{0,m}\big\}$, and $\M G=\M C_{\VRand Y_0} \in \R^{M \times M}$ with $[\M G]_{m,n}=\langle \nu_m,\Op A \nu_n \rangle=\Exp\big\{\Rand y_{0,m}\hspace{.08em}\Rand y_{0,n}\big\}$.
Since the noise has zero mean and is independent of the signal (and, hence of the noise-free measurements $\VRand Y_0$ of $\gRand G$), we also have that $\Exp\big\{\Rand x\hspace{.08em}\Rand y_{0,m}\big\}=\Exp\big\{\Rand x\hspace{.08em}\Rand y_{m}\big\}$ and $\VRand Y\sim \Spc N(\V 0, \M C_{\VRand Y_0}+\sigma^2 \M I)$. This allows us to conclude that
 $\VRand z\sim \Spc N(\V 0, \M C_{\VRand z})$ with
\begin{align}
\label{Eq:AugmentedCovarZ}
\M C_{\VRand z}=\left(\begin{array}{cc}
\sigma_\Rand x^2  &   \M c^\Tr    \\[0.5ex]
\M c  &   \M G + \sigma^2 \M I 
\end{array}
\right).
\end{align}
By remenbering that $\langle \Op A \varphi, \nu_m\rangle=\langle \varphi^\ast, \nu_m\rangle_{\Spc H'}=\langle \nu^\ast_m,\varphi
\rangle_{\Spc H}$, we further identify the covariance vector in \eqref{Eq:AugmentedCovarZ} as $\M c=\V \nu^\ast(\varphi)=\V \nu(\Op A \varphi)$, where $\V \nu^\ast: \Spc H' \to \R^M$ 
is the representer in $\Spc H$ of the measurement operator $\V \nu$. The other remarkable aspect is that $\M G$ in \eqref{Eq:AugmentedCovarZ} is the same matrix as in \eqref{Eq:GSmoothingSpline}

We are now in familiar territory and can invoke a standard result in multivariate statistics:
If two sets of variables are jointly Gaussian, then the distribution of one set conditioned on the other is a multivariate Gaussian whose mean and covariance are in direct relation with the partitioned mean vector and covariance matrix of the joined random vector (see Equations (2.81)-(2-82) p. 87 \cite{Bishop2006}). When applied to our setting, this tells us that
$
p_{\Rand X|\VRand Y=\M y}
$
is univariate Gaussian with mean
\begin{align}
\Exp\{\Rand X|\M y\}=\mu_\Rand X + \M c^\Tr \M C^{-1}_{\VRand Y}(\M y-\V \mu_{\VRand Y})= \V \nu^\ast(\varphi)^\Tr (\M G + \sigma^2 \M I)^{-1} \M y
\end{align}
and variance
\begin{align}
\sigma^2_{\Rand X|\M y}=\sigma_\Rand X^2-   \M c^\Tr \M C^{-1}_{\VRand Y}\M c=\langle \Op A \varphi,\varphi\big\rangle + \V \nu^\ast(\varphi)^\Tr  (\M G + \sigma^2 \M I)^{-1}   \V \nu^\ast(\varphi). 
\end{align}
Moreover, we know that the conditional mean $\Exp\{\Rand X|\VRand Y=\M y\}$ is the MMSE estimate of $\Rand X$ given $\M y$---a result that holds for any distribution.

We now summarize the outcome of this derivation in relation to our initial signal-recovery problem.
\begin{theorem}[Generalized Tikhonov-Wahba-Wiener theorem] 
\label{Theo:GenGaussMarkov}
We consider the following setting.
\begin{itemize}
\item The signal $g=\{g(\varphi)\}_{\varphi \in \Spc S} \in \Spc S'$ to be recovered  is a realization of a generalized Gaussian process $\gRand g\sim \Spc N(0,\Op A) $ in $\Spc S'$.
\item There exists an invertible whitening/regularization operator $\Lop$ such that $\Op A=(\Lop^\ast \Lop)^{-1}: \Spc S \to \Spc S'$ is positive-definite. This operator induces a native Hilbert space $\Spc H=\{f\in \Spc S': \|\Lop f\|_{L_2}<\infty\}$ whose continuous dual $\Spc H'$ is equipped with the inner product $\langle \nu_1,\nu_2\rangle_{\Spc H'}=\langle \Op A\Op \nu_1,\nu_2\rangle$.
\item The observed data are $\M y=\V \nu(g) + \V \epsilon \in \R^M$,
with $\V \nu=(\nu_1, \dots,\nu_M)\in (\Spc H')^M$ and $\V \epsilon$ a realization of a white Gaussian noise  with variance $\sigma^2$.
\end{itemize}
Then, for any $\varphi\in \Spc H'\supset \Spc S$, the MMSE estimator of
$g(\varphi)$ given $\M y$ 
is 
\begin{align}
\label{Eq:MMSEsol0}
g_{{\rm MMSE}}(\varphi|\M y)=\Exp\{\gRand g( \varphi)|\M y\}
&=g_{\lambda}( \varphi)\quad \mbox{ with }\quad  \lambda=\sigma^2,
\end{align}
where 
$g_{\lambda}\in \Spc H$ is the generalized smoothing spline defined by \eqref{Eq:GSmoothingSpline}.

\end{theorem}

Since the result in Theorem \ref{Theo:GenGaussMarkov} holds for any $\varphi\in  \Spc S$, it gives a precise meaning to the statement that the generalized spline $g_\lambda=\{g_\lambda(\varphi): \varphi \in \Spc S\} \in  \Spc S'$ with $\lambda=\sigma^2$ is the optimal estimate of the generalized Gaussian process
$\gRand G=\{\gRand g(\varphi): \varphi \in \Spc S\}$ in $\Spc S'$ given the measurements $\M y$. As for the case where $\Spc H$ is a RKHS of ordinary functions $f: \R^d \to \R$, we have that $\delta_{\V x_0} \in \Spc H'$ for any 
$\V x_0 \in \R^d$. Consequently, we can plug $\varphi=\delta_{\V x_0}$ in \eqref{Eq:MMSEsol0}, which yields Wahba's equivalence for classic  Gaussian processes $\gRand G=\{\gRand g(\V x)\}_{\V x \in \R^d}$ whitened by $\Lop$; namely, that $\Exp\{\gRand g(\V x_0)|\M y\}=g_{\lambda}(\V x_0)$
with $\lambda=\sigma^2$.

We shall now state two variants of Theorem \ref{Theo:GenGaussMarkov} for GGPs that are whitened by a $\V p$-admissible operator $\Lop$.
The first is the direct generalization of a weaker variant of \eqref{Eq:MMSEsol0} for ordinary stochastic processes that states that $\Lop$-smoothing splines yield
the {\em best linear unbiased} (BLU) estimator of a stochastic process $\{\gRand X(t)\}$ whitened by $\Lop$ under the assumption of a fixed-effect model that includes an additional deterministic drift $p(\cdot) \in \Spc N_\V p$ \cite{Wahba1990,Berlinet2004}. In our formalism, the generalized version of this model reads
\begin{align}
\label{Eq:FixedEffect}
\begin{cases}\gRand G&=\sum_{n=1}^{N_0}  \theta_n p_n + \gRand X\\
\Rand Y_m&=\gRand G(\nu_m) + \epsilon_m,\quad m=1,\dots,M,
\end{cases}
\end{align}
where $\V \theta=(\theta_n)$ is fixed but unknown, $\V \epsilon=(\epsilon_m)\sim \Spc N(0,\sigma^2\M I)$, and where $\Rand X\sim \Spc N(0,\Op A_{\V \phi})$ is a GGP
whose covariance operator $\Op A_{\V \phi}: \Spc S \to \Spc S'$ is given by \eqref{Eq:Apadmis} with $\Op A=(\Lop^\ast \Lop)^{-1}$ where $\Lop$ is ${\V p}$-admissible with $\V p=(p_1, \dots,p_{N_0})$.

The BLU estimator of $\gRand G(\varphi)$ in \eqref {Eq:FixedEffect} given the observed data $\VRand Y=\M y$ has the form
$g_{\rm BLU}(\varphi|\M y)=\V \beta^\Tr \M y$ (linearity), with its coefficient-vector $\V \beta=\V \beta(\varphi) \in \R^M$ being selected such as to minimize $\Exp\{\big(g_{\rm BLU}(\varphi|\M y)- \gRand G(\varphi)\big)^2\}$ subject to the zero-bias constraint $\Exp\{g_{\rm BLU}(\varphi|\M y)- \gRand G(\varphi)|\V \theta\}=0$ for all $\varphi \in \Spc H'$. The relation of this estimator to splines is stated in Theorem \ref{Theo:WahbaKimeldorf}, with the proof
being exactly the same as in the classic case.
\begin{theorem}[Wahba-Kimeldorf \cite{Kimeldorf1971}]
\label{Theo:WahbaKimeldorf}
Let $g_\lambda$ be the generalized smoothing spline defined by
\begin{align}
g_\lambda=\arg \min_{g \in \Spc H}\left( \|\M y - \V \nu(g)\|^2_2 + \lambda \|\Lop g\|^2_{L_2}\right).
\end{align}
Then, for any $\varphi \in \Spc H'$ and $\lambda=\sigma^2$, $g_\lambda(\varphi)=g_{\rm BLU}(\varphi|\M y)$ where $g_{\rm BLU}(\varphi|\M y)$ is the {\rm best linear unbiased estimator} of $\gRand G(\varphi)$ in \eqref {Eq:FixedEffect} given $\M y$.
\end{theorem}

Our final result fills the gap that was stated in the introduction by showing that every generalized spline interpolant is the MMSE
estimator of some GGP. The condition that makes this possible is $\Spc N_{\V \phi} \subset \Spc N_{\V \nu}$, which means that the observations of the process must be rich enough to span the boundary conditions $\V \phi(\gRand G)=\V 0$, which are essential in our theory.
\begin{theorem}[MMSE interpolation of a GGP]
\label{Theo:Unser}
Let $\Rand Y_m=\gRand G(\nu_m), m=1,\dots,M$ be a finite collection of observations of a generalized Gaussian process $\Rand G\sim \Spc N(0,\Op A_{\V \phi})$ in $\Spc S'$ that is whitened by a $\V p$-admissible operator $\Lop$. Then, under the condition that $\Spc N_{\V \phi} \subset {\rm span}\{ \nu_m\}_{m=1}^{M}\subset \Spc H'$, the {\em MMSE interpolator} of $\gRand G$ given the observed data
$\VRand y=\M y$ is the generalized spline interpolant 
\begin{align}
g_0=\arg \min_{g \in \Spc H} \|\Lop g\|_{L_2}
\quad \mbox{s.t.}\quad \V \nu(g_0)=\M y.
\end{align}
It is such that $\Exp\{\gRand G(\varphi)|\M y\}=g_0(\varphi)$ for any $\varphi \in \Spc H'\supset \Spc S$.
\end{theorem}
To the best of our knowledge, Theorem \ref{Theo:Unser} is new (at least with the present level of generality), the very reason being that its formulation necessitates a proper definition of the corresponding GGP, which had been missing until now. 

We conclude the chapter by emphazing the crucial role of the regularization/whitening operator $\Lop$ and of its native space $\Spc H_\Lop$. As for the deterministic part of the story, the generalized forms of $\Lop$-splines are the solutions of regularized linear inverse problems in $\Spc H_\Lop$. Their stochastic counterparts are precisely the GGPs whitened by $\Lop$. (The delicate part in this statement is the necessity to invert $\Lop$, which requires the specification of boundary conditions.) The pleasing outcome is the mathematical congruence between the deterministic and stochastic worlds. The present framework provides a rigorous support to statements such as ``fractional splines are optimal estimators for fractional Brownian motion.'' Likewise, we can now identify the polynomial splines of degree $2n-1$ as the optimal estimators of Gaussian processes, which, up to some boundary conditions (the technical part of the story), are the $n$-fold integration of a white noise, or, for more classic folks,  the $(n-1)$-fold integration of a Brownian motion.
With the help of abstraction and nuclear spaces, we ultimately end up with an equivalence in the continuum that is as strong as in finite dimensions.
\appendix
\section{Appendix: The Finite-Dimensional Setting}
Many estimation/prediction problems in statistics and signal processing as well as inverse problems in science and engineering can be reduced to the recovery of an unknown signal, represented by a vector $\M x \in \R^N$, from a set of ``noisy'' linear measurements $\M y \in \R^M$. This setting is generically described by the measurement model 
\begin{align}
\label{Eq:linmodel}
\M y= \M H \M x + \V \epsilon,
\end{align}
where $\M H=[\M h_1 \cdots \M h_M]^\Tr \in \R^{M \times N}$ is a known ``system matrix'' that models the physics of the acquisition/measurement process and $\V \epsilon$ is some additive noise or disturbance term. We now briefly recall the classic variational and statistical formulations of such problems and show that the corresponding signal reconstructions can be deduced as special cases of Theorems \ref{Theo:GeneralRepSemiHilbert} and \ref{Theo:GenGaussMarkov} with $\Spc S=\Spc S(\mathbb{I})=\R^N$ where $\mathbb{I}=\{1,\dots,N\}$. 

\subsection{Variational or Tikhonov Solution}
When the inversion of \eqref{Eq:linmodel} is ill-posed, the standard approach to estimate $\M x$ is to perform a regularized-least-square fit of the data. 
In the Tikhonov framework, this is formulated as
\begin{align}
\label{Eq:Tikformulation}
\M x_{\lambda}&=\arg \min_{\M x \in \R^N} \left( \|\M y - \M H \M x\|_2^2 + \lambda \|\M L \M x\|_2^2\right),
\end{align}
where $\M L: \R^N \to \R^N$ is an appropriate regularization operator and $\lambda\in \R_{\ge0}$
an adjustable parameter that controls the strength of the regularization. By setting the gradient of the Tikhonov loss to zero, one gets the classic linear solution
\begin{align}
\M x_{\lambda}
&=
(\M H^\Tr \M H + \lambda \M L^\Tr\M L)^{-1} \M H^\Tr \M y.
\label{Eq:Tikhonov}
\end{align}
The key now is that, when $\M L$ is invertible (or, equivalently, when the symmetric matrix $\M A=(\M L^\Tr \M L)^{-1}$ is positive-definite), we can rewrite \eqref{Eq:Tikhonov} as
\begin{align}
\M x_{\lambda}
&=
\M H^\Tr\M A ( \M G+ \lambda \M I)^{-1} \M y \quad  \mbox{ with }\quad 
\M G= \M H \M A\M H^\Tr.
\label{Eq:Tikhonov2}
\end{align}
This can be readily verified if one forms the difference of the reconstruction matrices in \eqref{Eq:Tikhonov} and \eqref{Eq:Tikhonov2} and gets rid of the inverses by suitable right and left multiplication. 

By comparing \eqref{Eq:Tikhonov2} to \eqref{Eq:TikAbstract}, we identify \eqref{Eq:Tikhonov} as a particular case of
Theorem \ref{Theo:GeneralRepSemiHilbert} with $N_0=0$,  $\Spc H=(\R^N,\|\cdot\|_2)$,  $f=\M x$, $\nu_m=\M h_m$, and $\psi_m= \M A \M h_m$.
This is the simplest functional setting of our theory. There, all the underlying vector spaces are topological equivalent: $\Spc S(\mathbb{I})\simeq\Spc S'(\mathbb{I})\simeq
L_2=\ell_2(\mathbb{I})=(\R^N,\|\cdot\|_2)$. The required completeness of $\Spc H$ and the continuity of $\Op L: \Spc H \to L_2$  are automatically satisfied because: (i)  all finite-dimensional Hilbert spaces are isomorphic to $\ell_2(\mathbb{I})=\R^N$ (for short);
and (ii) any linear operator $\ell_2(\mathbb{I})\to\ell_2(\mathbb{I})$ is de facto continuous and can be implemented as a matrix multiplication.

\subsection{MMSE Solution under the Gaussian Hypothesis}
We now consider a measurement model that is the stochastic counterpart of \eqref{Eq:linmodel}. There,
the signal $\M x$ and the noise $\V \epsilon$ are realizations of Gaussian random
vectors $\VRand X \sim \Spc N(\V \mu_{\VRand x},\M C_{\VRand x})$ and $\VRand N \sim \Spc N(\M 0,\sigma^2 \M I)$, respectively. The measurement noise is zero-mean i.i.d.\ with variance $\sigma^2$ and independent of the signal. As for the signal $\M x$, one can obtain
its prior pdf by taking the $N$-dimensional inverse Fourier transform of the Gaussian characteristic function \eqref{Eq:MultiGaussFourier}. This yields
\begin{align}
\label{eq:multiGauss}
p_\VRand x(\M x)=\frac{1}{\sqrt{(2 \pi)^{N} |{\rm det}(\M C_{\VRand x})|}} \exp\left(-\tfrac{1}{2}(\M x-\V \mu_{\VRand x})^\Tr \M C_{\VRand x}^{-1} (\M x-\V \mu_{\VRand X})\right),
\end{align}
which is the generic form of an $N$-dimensional multivariate Gaussian probability density function with mean $\V \mu_{\VRand X}=\Exp\{\VRand X\} \in \R^N$ and covariance matrix 
$\M C_{\VRand x} \in \R^{N \times N}$. The latter is required to be positive definite (and, hence, invertible) for \eqref{eq:multiGauss} to be well-defined.

In the stochastic setting,  the MMSE estimator of $\VRand X$ given $\M y$ is the optimal reconstruction of the signal $\M x$. The foundational result in estimation theory is that, under the above Gaussian hypothesis, the reconstruction is given by
the so-called Wiener estimator
\begin{align}
\label{eq:Wiener}
\M x_{\rm MMSE}=\Exp\{\VRand X|\VRand Y=\M y\}=\V \mu_{\VRand X}+\M C_{\VRand x} \M H^\Tr (\M H \M C_{\VRand x}\M H^\Tr + \sigma^2 \M I)^{-1}\M y.
\end{align}
For $\V \mu_{\VRand X}=\V 0$, this estimator has the same linear form as \eqref{Eq:Tikhonov2}
with $\M C_{\VRand x}=\M A=(\M L^\Tr \M L)^{-1}$ which, in view of the identifications made in Section A.2, confirms its equivalence with the Tikhonov reconstruction \eqref{Eq:Tikhonov}---the discrete version of a spline.
We can derive this Wiener estimator as a particular case of \eqref{Eq:MMSEsol0} in Theorem \ref{Theo:GenGaussMarkov} by taking the ``test functions'' $\{\varphi_n\}$ to be the canonical basis $\M e_1, \dots, \M e_N$ of $\R^N$, which then yields its component-wise description.
\subsection{Innovation Model and Equivalence with MAP Estimator}
\label{Sec:White}
The idea here is to ``standardize'' $\VRand x \sim \Spc N(\V \mu_{\VRand x},\M C_{\VRand x})$ by applying the affine transformation $$\VRand w=\M L (\VRand x-\V \mu_{\VRand x})\sim \Spc N(\M 0,\M I)$$ with a proper choice of the ``whitening'' operator $\M L$ : $\R^N\to\R^N$ such that $\M L \M C_{\VRand x}\M L^\Tr=\M I$. A standard choice 
is the symmetric square-root inverse\footnote{Alternatively, one may also consider the (unique) Cholesky factorization
of the Hermitian symmetric matrix $\M A=\M C_{\M x}^{-1}=\M L \M L^\HTop$, where $\M L$ is lower triangular, which is better suited for recursive Kalman-type implementations of the estimator.} of $\M C_{\VRand x}$, which is given by
$$\M L=\M C_{\VRand x}^{-1/2}=\sum_{n=1}^N \frac{1}{\sqrt{\lambda_n}} \M u_n\M u^\Tr_n$$
where $\lambda_1\ge \cdots\ge \lambda_N>0$ are the eigenvalues of $\M C_{\VRand x}$ with corresponding (orthonormal) eigenvectors $\M u_1,\dots,\M u_N$.  By setting
$\M w=\M L (\M x-\V \mu_{\VRand x})$, this enables us to rewrite the argument of the exponential in \eqref{eq:multiGauss}
 as
 \begin{align}
 \label{eq:innovalikelihood}
-\tfrac{1}{2}(\M x-\V \mu_{\VRand x})^\Tr \M C_{\VRand x}^{-1} (\M x-\V \mu_{\VRand X})=-\tfrac{1}{2}\|\M L (\M x -\V \mu_{\VRand x}) \|_2^2=-\tfrac{1}{2}\|\M w\|_2^2,
\end{align}
which is the innovation form of the log-prior. 

To derive the MAP estimator of 
$\VRand X$ given the noisy linear measurements $\M y$, we invoke Bayes rule $p_{\VRand X|\VRand Y=\M y}(\M x)\propto p_{\VRand Y|\VRand X=\M x}(\M y) p_{\VRand X}(\M x)$
with $p_{\VRand Y|\VRand X=\M x}(\M y)=p_{\VRand N}(\M y - \M H \M x)$. By switching to log-likelihoods and using \eqref{eq:innovalikelihood} to simplify the Gaussian prior, we 
find that this estimator is given by 
\begin{align}
\label{Eq:MAP}\M x_{\rm MAP}&=\arg \max_{\M x \in \R^N} p_{\VRand X|\VRand Y=\M y}(\M x)\nonumber\\
&=\arg \min_{\M x \in \R^N} \left(\frac{1}{\sigma^2} \|\M y - \M H \M x\|_2^2 +  \|\M L (\M x-\V \mu_{\VRand x})\|_2^2\right).
\end{align}
The crucial observation is that \eqref{Eq:MAP} with $\mu_{\VRand x}=\V 0$ and $\sigma^2=\lambda$ is equivalent to \eqref{Eq:Tikformulation}. By determining the solution of \eqref{Eq:MAP}, which is an affine variant of \eqref{Eq:Tikhonov}, we also confirm that $\M x_{\rm MAP}=\M x_{\rm MMSE}$ where $\M x_{\rm MMSE}$ results from the Wiener filter found in \eqref{eq:Wiener}. (The verification procedure is the same as for the alternative form \eqref{Eq:Tikhonov2} of the Tikhonov estimator.)

\subsection*{Acknowledgment}
The research described in this review was supported by the Swiss National Science Foundation (SNSF) under Grant 200020\_219356. Part of this material was presented at the Wiener keynote lecture at ICASSP 2026 in Barcelona, Spain.


\bibliographystyle{ieeetr}


%
%
\bibliography{Kailath.bib}



\end{document}